\documentclass{amsart}
\usepackage{amsfonts,  bbm}
\usepackage{overpic}

\newtheorem{theorem}{Theorem}[section]
\theoremstyle{plain}

\newtheorem{lemma}[theorem]{Lemma}
\newtheorem{prop}[theorem]{Proposition}
\theoremstyle{remark}

\numberwithin{equation}{section}

\newcommand{\tr}{\operatorname{tr}}

\newcommand{\re}{\operatorname{Re}}
\newcommand{\im}{\operatorname{Im}}
\newcommand{\res}{\operatorname{Res}}
\newcommand{\rank}{\operatorname{rank}}

\newcommand{\supp}{\operatorname{supp}}
\newcommand{\dist}{\operatorname{dist}}

\newcommand{\bbR}{\mathbb{R}}
\newcommand{\bbH}{\mathbb{H}}
\newcommand{\bbB}{\mathbb{B}}
\newcommand{\bbC}{\mathbb{C}}
\newcommand{\bbZ}{\mathbb{Z}}
\newcommand{\bbN}{\mathbb{N}}
\newcommand{\calR}{\mathcal{R}}

\newcommand{\calW}{\mathcal{W}}

\newcommand{\cinf}{C^\infty}
\newcommand{\del}{\partial}

\newcommand{\vep}{\varepsilon}

\newcommand{\nh}{\tfrac{n}{2}}
\newcommand{\bQ}{\mathbf{Q}}
\newcommand{\Ai}{{\rm Ai}}
\newcommand{\norm}[1]{\left\Vert #1 \right\Vert}
\newcommand{\abs}[1]{\left|#1 \right|}
\newcommand{\brak}[1]{\left\langle #1 \right\rangle}
\newcommand{\chr}{\mathbbm{1}}
\newcommand{\tN}{\widetilde{N}}

\begin{document}

\title{Resonance asymptotics for Schr\"odinger operators on hyperbolic space}
\author[Borthwick]{David Borthwick}
\address{Department of Mathematics and Computer Science, Emory
University, Atlanta, Georgia, 30322, U. S. A.}
\thanks{Borthwick supported in part by NSF\ grant DMS-0901937.}
\email{davidb@mathcs.emory.edu}
\author[Crompton]{Catherine Crompton}
\address{Department of Mathematics and Computer Science, Emory
University, Atlanta, Georgia, 30322, U. S. A.}
\email{lcrompt@mathcs.emory.edu}
\subjclass[2000]{Primary 58J50,35P25; Secondary 47A40}
\date{\today}

\begin{abstract}
We study the asymptotic distribution of resonances for scattering by compactly supported potentials in $\bbH^{n+1}$.  We first establish an upper bound for the resonance counting function that depends only on the dimension and the support of the potential.   We then establish the sharpness of this estimate by proving the a Weyl law for the resonance counting function holds in the case of radial potentials vanishing to some finite order at the edge of the support.
As an application of the existence of potentials that saturate the upper bound, we derive additional resonance asymptotics that hold in a suitable generic sense.  These generic results include asymptotics for the resonance count in sectors. 
\end{abstract}

\maketitle
\tableofcontents

\bigbreak
\section{Introduction}\label{intro.sec}

In this paper we will study the spectral asymptotics of Sch\"odinger operators of the form $\Delta + V$, where $\Delta$ is the (positive) Laplacian on $\bbH^{n+1}$, and $V\in L^\infty_{\rm cpt}(\bbH^{n+1},\bbC)$ is a compactly supported, possibly complex-valued potential.  The essential spectrum
of $\Delta +V$ is $[n^2/4,\infty)$ and is absolutely continuous.  The eigenvalue spectrum is finite and contained in $(0,n^2/4)$, and hence it is the resonance set that plays the role of discrete spectral data in this setting.

To define resonances, consider first the resolvent of $\Delta$, written in the form $R_0(s) = (\Delta - s(n-s))^{-1}$ for $\re s > \nh$.
The well-known formula expressing the kernel of $R_0(s)$ in terms of hypergeometric functions (see \cite{Patterson:1975}) shows immediately that the cutoff resolvent $\chi R_0(s) \chi$ admits a meromorphic extension to $s\in\bbC$, with poles of finite rank, for any $\psi \in \cinf_0(\bbH^{n+1})$.

It is easy to extend this meromorphic continuation result to $R_V(s) := (\Delta + V - s(n-s))^{-1}$; see \S\ref{potsc.sec} for the details.
We define the resonance set $\calR_V$ as the set of poles of $R_V(s)$, counted according to multiplicity given by the rank of the residue.
\begin{figure}
\begin{center}
\begin{overpic}[scale=.8]{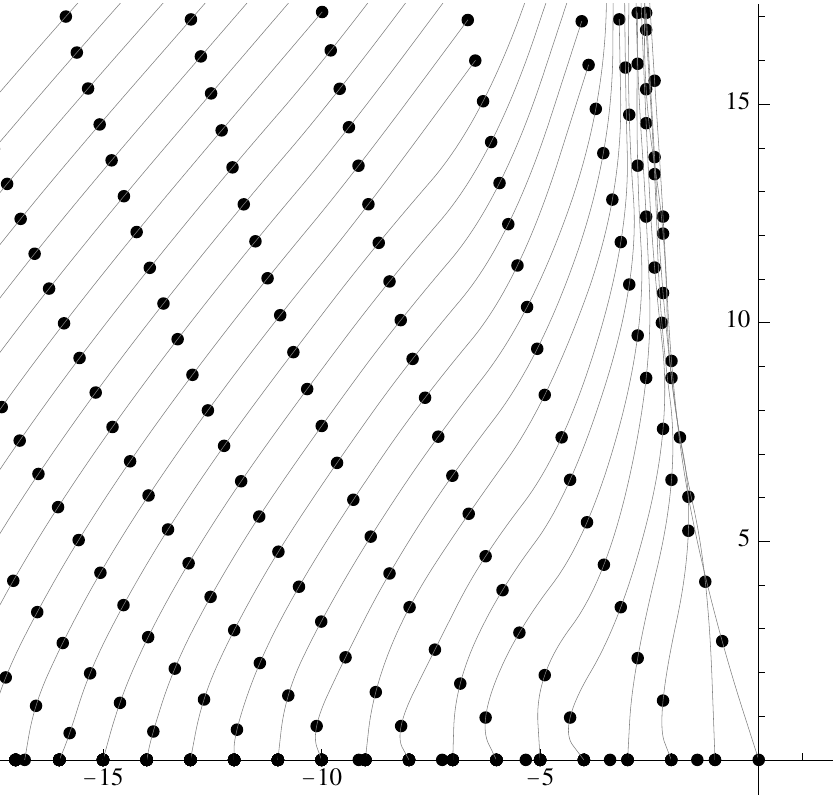}
\end{overpic}
\end{center}
\caption{Resonance plot for a radial step potential in $\bbH^{3}$.  The thin lines indicate the spherical harmonic mode $l$,
with $l =0$ at the right.  The multiplicity on each line is $2l+1$.
(Each mode also contributes resonances at negative integers, not on these lines.)}\label{RPotRes3.fig}
\end{figure}
The resonance counting function is 
\[
N_V(t) := \#\left\{\zeta\in\calR_V:\>\abs{\zeta-\nh} \le t \right\}.
\]
Figure~\ref{RPotRes3.fig} shows a sample of the resonance set for $V = \chi_{B(1)}$ in $\bbH^3$, the characteristic function of the unit ball.  The corresponding counting function shown in Figure~\ref{NV3.fig}.   These plots are based on explicit calculation of the resonance set in terms of Legendre functions; see \S\ref{scmatrix.sec} for the formulas.

It is essentially already known that
\begin{equation}\label{NV.bound}
N_V(t) = O(t^{n+1}).
\end{equation}
That is, for real $V$ this is a special case of \cite[Thm.~2.2]{Borthwick:2010}, and the extension to compactly supported complex potentials is straightforward.
\begin{figure}
\begin{center}
\begin{overpic}[scale=.8]{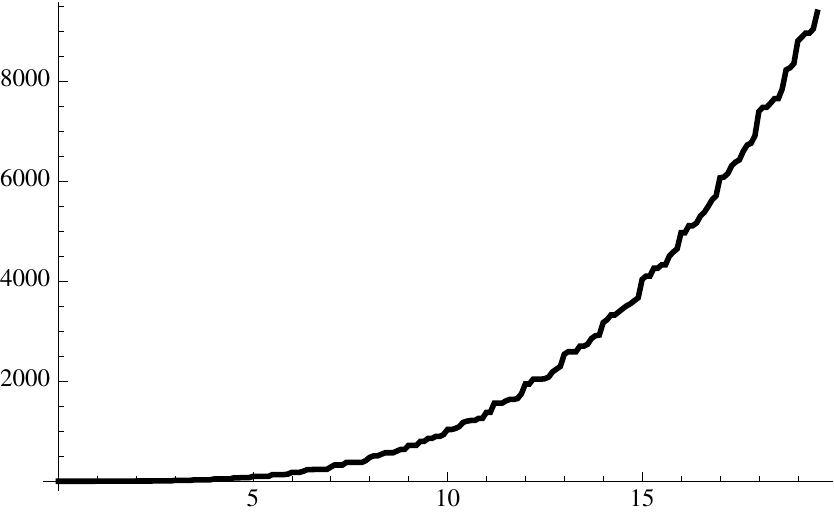}
\put(72,40){$N_V(t)$}
\end{overpic}
\end{center}
\caption{Resonance counting function for a radial step potential in $\bbH^3$.}\label{NV3.fig}
\end{figure}

In the case $V=0$, the resonance set is of course well-known: for $n$ odd we have
\[
\calR_0 = \left\{ \zeta \in - \bbN_0:\>m_0(-k) = (2k+1) \frac{(k+1)\cdots(k+n-1)}{n!}\right\},
\]
while for $n$ even there are no resonances, $\calR_0 = \emptyset$.   In the case of $n$ odd we thus have a simple asymptotic, 
\begin{equation}\label{N0.asymp}
N_0(t) = \frac{2}{(n+1)!} t^{n+1} + O(t^n).
\end{equation}
For later reference, we label the asymptotic constant for the model case as
\begin{equation}\label{A0.def}
A^{(0)}_n := \begin{cases}
\frac{2}{(n+1)!} & n \text{ odd},\\
0 & n \text{ even}.
\end{cases}
\end{equation}

For more general families of compactly supported perturbations of the Laplacian on $\bbH^{n+1}$, including metric and topological perturbations as well as smooth real potentials, sharp estimates of $N(t)$ were obtained in Borthwick \cite{Borthwick:2010}.  
Our first result is an extension of this bound to non-smooth, complex-valued potentials.  
These estimates involve the integrated version of the counting function,
\[
\tN_V(a) :=  (n+1)\int_0^a \frac{N_V(t) - N_V(0)}{t} dt,
\]
a common usage in the theory of entire functions.

The constant in the upper bound is expressed in terms of a \emph{indicator function}, defined for $\abs{\theta} \le \tfrac{\pi}2$ by
\begin{equation}\label{indicator.def}
h_{r_0}(\theta) :=  \frac{2}{\Gamma(n)}  \int_0^\infty  
\frac{[H(x e^{i\theta}; r_0)]_+}{x^{n+2}}\>dx,
\end{equation}
where $[\cdot]_+$ denotes the positive part and
\begin{equation}\label{H.def}
\begin{split}
H(\alpha, r) & :=   \re \left[ 
2\alpha \log \Bigl( \alpha \cosh r +  \sqrt{1 + \alpha^2 \sinh^2 r} \Bigr) - \alpha \log (\alpha^2-1) \right] \\
&\qquad + \log \left| \frac{\cosh r - \sqrt{1 + \alpha^2 \sinh^2 r}}{\cosh r + \sqrt{1 + \alpha^2 \sinh^2 r}} \right|.
\end{split}
\end{equation}
The corresponding asymptotic constant is the integral,
\begin{equation}\label{An.def}
A_n(r_0) := A^{(0)}_n + \frac{n+1}{2\pi} \int_{-\frac\pi2}^{\frac\pi2} h_{r_0}(\theta)\>d\theta.
\end{equation}
(The angular dependence of the indicator function will play a role later when we consider the distribution of resonances in sectors.)

\begin{theorem}\label{upper.thm}
Suppose that $V \in L^\infty_{\rm cpt}(\bbH^{n+1}, \bbC)$ has support contained in a closed ball of radius $r_0$.
Then 
\[
\tN_V(a) \le A_n(r_0) a^{n+1} + O(a^n \log a).
\]
\end{theorem}

The error estimate in Theorem~\ref{upper.thm} improves on the $o(a^{n+1})$ result of \cite{Borthwick:2010}, but this improvement is specific to the potential case.  It is based an sharper scattering phase estimate that we will give in Proposition~\ref{sigma.bound}. 

Our main goal in this paper is to demonstrate the sharpness of Theorem~\ref{upper.thm} in the case of radial potentials.  This is analogous to the Euclidean result for radial potentials in odd dimensions due to Zworski \cite{Zworski:1989}, with the exact constant later computed by Stefanov \cite{Stefanov:2006}.

\begin{theorem}\label{asymp.thm}
Suppose that $V \in L^\infty(\bbH^{n+1}, \bbC)$ is a radial potential with support in a ball of radius $r_0$.  If we assume that
$V$ is continuous near $r=r_0$ and has finite order of vanishing in the sense that
\[
V(r) \sim \kappa (r_0-r)^{\beta},\quad\text{as }r\to r_0,
\]
for some $\beta\ge 0$ and $\kappa \ne 0$.  Then
\[
N_V(t) \sim A_n(r_0) t^{n+1}.
\]
\end{theorem}

\bigbreak
In the final section of the paper we include some applications of this result.  In the Euclidean case, Christiansen \cite{Christ:2005, Christ:2012}
has established generic properties of resonance distributions for potential scattering, and the exact asymptotic for the radial case plays a key role in this work.  Using Theorem~\ref{asymp.thm} we can prove some analogous theorems for $\bbH^{n+1}$.  These results (see \S\ref{dist.sec} for the precise statements) include:
\begin{enumerate}
\item  For generic (real or complex) potentials $V$ with support in a compact set $K\subset \bbH^{n+1}$ with non-empty interior, the order of growth of the resonance counting function is optimal in the sense that
\[
\limsup_{t\to\infty} \frac{\log N_V(t)}{\log t} = n+1.
\]
\medskip
\item   For generic (real or complex) potentials supported in the closed ball $\overline{B}(r_0)$, 
\[
\limsup_{t\to\infty} \frac{N_V(t)}{t^{n+1}} = A_n,
\]
i.e.~the estimate in Theorem~\ref{upper.thm} is generically optimal for potentials with support equal to $\overline{B}(r_0)$.
\medskip
\item   For generic (real or complex) potentials with support in the closed ball $\overline{B}(r_0)$, there is a generic lower bound on the number of resonances contained in a sector near the critical line, with the optimal order of growth $n+1$ and a constant independent of the size of the sector. 
\medskip
\item  For potentials with support in $\overline{B}(r_0)$ for which $N_V(t) \sim A_n(r_0)t^{n+1}$, the asymptotic distribution of resonances in sectors is governed by the indicator function $h_{r_0}(\theta)$ defined in \eqref{indicator.def}.  The same distribution holds, in a weighted average sense, for families of perturbations of of such potentials.
\end{enumerate}

The paper is organized as follows.  In \S\ref{potsc.sec} we introduce the basic spectral operators associated to $\Delta+V$, the scattering matrix in particular.  In \S\ref{rescount.sec} we establish the formula for counting resonances in terms of the relative scattering determinant, which is the basis for the rest of the analysis.  We also prove some general estimates on the scattering determinant that will be needed later, and which in particular give the proof of Theorem~\ref{upper.thm}.  Explicit formulas for scattering matrix elements in the radial case are worked out in \S\ref{scmatrix.sec}.
In \S\ref{asym.sec} we develop precise recursive estimates for these matrix elements.  From these estimates we derive asymptotics
of the scattering determinant in \S\ref{rscdet.sec}, yielding the proof of Theorem~\ref{asymp.thm}.  Finally \S\ref{dist.sec} contains the  resonance distribution results for generic potentials as outlined above. 

\vskip12pt\noindent
\textbf{Acknowledgment.}  We would like to thank Tanya Christiansen for advice on the proof of Lemma~\ref{sigma.bound} as well as for helpful discussions related to the material in \S\ref{dist.sec}.

\bigbreak
\section{Potential scattering in $\bbH^{n+1}$}\label{potsc.sec}

Consider a Schr\"odinger operator $\Delta+V$ in $\bbH^{n+1}$, with potential $V \in L^\infty_{\rm cpt}(\bbH^{n+1}, \bbC)$.
The resolvent $R_V(s) := (\Delta + V - s(n-s))^{-1}$ is defined by the spectral theorem for $\re s$ sufficiently large, and
is related to the model resolvent $R_0(s)$ by the identity
\begin{equation}\label{r0rv}
R_0(s) - R_V(s) = R_V(s)V R_0(s).
\end{equation}
As mentioned in the introduction, the cutoff resolvent $\chi R_0(s) \chi$ admits a meromorphic continuation to $s\in\bbC$.  We can describe this more precisely in terms of weighted $L^2$ spaces.  
In terms of geodesic polar coordinates $(r,\omega)$ for $\bbH^{n+1} \cong \bbR_+ \times S^n$, define $\rho := 2e^{-r}$.  In terms of the Poincar\'e ball compactification of $\bbH^{n+1}$, $\rho$ is a boundary defining function.
The model resolvent $R_0(s)$ extends meromorphically to $\re s > -N+\nh$, as an operator
$\rho^N L^2(\bbH^{n+1}) \to \rho^{-N} L^2(\bbH^{n+1})$.  Since $VR_0(s)$ is compact as an operator on $\rho^N L^2(\bbH^{n+1})$, for $\re s > -N+\nh$, with arbitrarily small norm for $\re s$ sufficiently large, the analytic Fredholm theorem yields a meromorphic inverse $(1 + VR_0(s))^{-1}$.  In conjunction with \eqref{r0rv}, this establishes the following:
\begin{prop}
The resolvent $R_V(s)$ extends meromorphically to $s\in \bbC$ as
\[
R_V(s) = R_0(s) (1 + VR_0(s))^{-1},
\]
with poles of finite rank.
For any $N>0$, $R_V(s)$ is bounded as an operator $\rho^N L^2(\bbH^{n+1}) \to \rho^{-N} L^2(\bbH^{n+1})$ for $\re s > -N+\nh$.
\end{prop}

With meromorphic continuation of $R_V(s)$ established, we define the resonance set $\calR_V$ as the set of poles of $R_V(s)$, counted with multiplicities
\begin{equation}\label{mv.def}
m_V(\zeta) := \rank \res_\zeta R_V(s).
\end{equation}

\bigbreak
\subsection{Resolvent estimate}

For later use we need an estimate on the cutoff resolvent in the physical plane.
\begin{prop}\label{res.est}
Suppose $V \in L^\infty_{\rm cpt}(\bbH^{n+1}, \bbC)$ and $\chi \in C^\infty_0(\bbH^{n+1})$, with $\chi = 1$ on $\supp V$.  
There exist $C>0$, $M>0$ such that for $\theta \in [-\tfrac{\pi}2,\tfrac{\pi}2]$ and $a \ge M$,
\[
\norm{\chi R_V(\nh + ae^{i\theta}) \chi} \le Ca^{-1}.
\]
Here $C$ depends only on the support of $\chi$, while $M$ depends on $\supp \chi$ and $\norm{V}_\infty$.
\end{prop}
\begin{proof} 
Set $\Omega := \supp V$, and let $\chr_\Omega$ be the projector given by multiplication by the characteristic function of $\Omega$.  Since $(1 - \chr_\Omega)V = 0$, we can write
\[
1 + VR_0(s) = (1+VR_0(s)(1-\chr_\Omega)) (1+VR_0(s)\chr_\Omega).
\]
And then by inverting this expression we have the identity 
\[
(1+ VR_0(s))^{-1} = (1+VR_0(s)\chr_\Omega)^{-1} (1-VR_0(s)(1-\chr_\Omega)).
\]
This allows us to write the cutoff resolvent as
\begin{equation}\label{cut.rv0}
\chi R_V(s) \chi = \chi R_0(s) \chi (1+VR_0(s)\chr_\Omega)^{-1} (1-VR_0(s)(1-\chr_\Omega)\chi)
\end{equation}
(This trick works just as in the Euclidean case; see, e.g., \cite{Zworski:Res}.)

To estimate the terms involving $R_0(s)$, we can cite Guillarmou \cite[Prop.~3.2]{Gui:2005c}, which gives the estimate
\begin{equation}\label{R0.est}
\norm{\rho^{\frac12} R_0(s) \rho^{\frac12}} \le C \abs{s - \nh}^{-1},
\end{equation}
for $\re s > \nh-\tfrac18$, $s \ne \nh$.   Applying this to the the cutoff resolvent gives
\[
\norm{\chi R_0(\nh + ae^{i\theta}) \chi} \le Ca^{-1}.
\]
We can also apply \eqref{R0.est} to obtain, for $a$ sufficiently large, the estimates
\[
\norm{VR_0(\nh + ae^{i\theta})\chr_\Omega} \le \tfrac12 \quad\text{and}\quad
\norm{VR_0(\nh + ae^{i\theta})(1-\chr_\Omega)\chi}  \le \tfrac12.
\]
The claim then follows from \eqref{cut.rv0}.
\end{proof}

\bigbreak
\subsection{Scattering theory}

The scattering matrix $S_V(s)$ associated to the potential $V$ can be defined in the same way as for any asymptotically hyperbolic manifold.  We will recall the details rather briefly; see \cite{Borthwick:2010} for details and references.
 
Given $s \notin \calR_V \cup (\nh+\bbZ)$ and $f\in\cinf(S^n)$, there is a unique solution of $[\Delta+V-s(n-s)]u=0$ with the asymptotic
\begin{equation}\label{scatt.asymp}
u \sim \rho^{n-s} f + \rho^{s} f',
\end{equation}
for some $f'\in\cinf(S^n)$.  The scattering matrix, defined as the map $S_V(s): f\mapsto f'$, is a meromorphic family of pseudodifferential operators on $S^n$.  Note that by construction, the scattering matrix satisfies $S_V(n-s) = S_V(s)^{-1}$.

The relative scattering matrix $S_V(s)S_0(s)^{-1}$ is of determinant class, and the relative scattering determinant is defined as
\begin{equation}\label{tau.def}
\tau(s) := \det S_V(s)S_0(s)^{-1}.
\end{equation}
The reflection formula for the scattering matrix implies $\tau(s) \tau(n-s) = 1$.
This meromorphic function admits a Hadamard factorization over the resonance sets $\calR_V$ and $\calR_0$:
\begin{equation}\label{tau.factor}
\tau(s) = e^{q(s)} \frac{H_V(n-s)}{H_V(s)}  \frac{H_0(s)}{H_0(n-s)},
\end{equation}
where $q(s)$ is a polynomial of degree at most $n+1$, and
\[
H_V(s) := \prod_{\zeta\in\calR_V} \left( 1- \frac{s}{\zeta}\right)  \exp\left(\frac{s}{\zeta} + \dots + \frac{s^{n+1}}{(n+1)\zeta^{n+1}}\right).
\]
For  $V$ real-valued, the factorization \eqref{tau.factor} is a special case of \cite[Prop.~3.1]{Borthwick:2010}, and the extension to complex $V$ is straightforward.

\bigbreak
\section{Resonance counting formula}\label{rescount.sec}
One consequence of \eqref{tau.factor} is that we can count resonances with a contour integral over the scattering determinant.
Integrating $\tau'/\tau$ around a half-circle contour centered at $s = \nh$ yields
\[
N_V(t) - N_V(0) - N_0(t) - 2d_V(t) = \frac{1}{2\pi}  \int_{-t}^t \im \frac{\tau'}{\tau}(it') \>dt'
+ \frac{1}{2\pi} \int_{-\frac{\pi}2}^{\frac{\pi}2} t\>\del_t \log \abs{\tau(\nh + te^{i\theta})}\>d\theta,
\]
where $d_V(t)$ counts the number of resonances $\zeta \in \calR_V$ with $\re s>\nh$ and $\abs{\zeta - \nh} \le t$, which occur only when $\zeta(n-\zeta)$ is a discrete eigenvalue.  There are only finitely many discrete eigenvalues, so $d_V$ is bounded.

If we divide this contour integral by $t$ and integrate, we obtain the relative counting formula,
\begin{equation}\label{relcount}
\begin{split}
\tN_V(a) - \tN_0(a) & = \frac{n+1}{2\pi} \int_0^a \int_{-t}^{t} \im \frac{\tau'}{\tau}(it') \>dt'\>\frac{dt}{t} \\
&\qquad + \frac{n+1}{2\pi} \int_{-\frac{\pi}2}^{\frac{\pi}2} \log \abs{\tau(\nh + a e^{i\theta})} d\theta + O (\log a).
\end{split}
\end{equation}
For a general self-adjoint perturbation of $\Delta$, of the type considered in \cite{Borthwick:2010}, 
the first integral could be expressed in terms of the \emph{scattering phase}
$\sigma(t) :=  \tfrac{i}{2\pi} \log \tau(\nh + it)$, which would be real-valued in that case.
In the case of a metric perturbation there is a Weyl law \cite[Cor.~3.6]{Borthwick:2010} giving the asymptotic $\sigma(t) \sim at^{n+1}$, with $a$ proportional to the volume of the perturbation.

For potential scattering we would expect the scattering phase term to be of lower order, and that is indeed the case.
\begin{prop}\label{sigma.bound}
For $V \in L^\infty_{\rm cpt}(\bbH^{n+1},\bbC)$, we have
$$
\abs{\frac{\tau'}{\tau}(\nh + it)} \le C_V\brak{t}^{n-1},
$$
for $t \in \bbR$ sufficiently large, where $C_V$ depends only on $\norm{V}_\infty$ and $\supp V$.  For real $V$ this gives the scattering phase estimate $\sigma(t) = O(\brak{t}^{n})$.
\end{prop}
We will defer the proof for a moment to observe the consequences for the resonance counting formula.  
Applying Proposition~\ref{sigma.bound} to \eqref{relcount}, and using the asymptotics for $N_0(t)$ given in \eqref{N0.asymp}, 
yields the following:
\begin{prop}\label{count.prop}
For $V \in L^\infty_{\rm cpt}(\bbH^{n+1}, \bbC)$, 
\[
\tN_V(a) = A^{(0)}_n a^{n+1} +  \frac{n+1}{2\pi} \int_{-\frac{\pi}2}^{\frac{\pi}2} \log \abs{\tau(\nh + a e^{i\theta})} d\theta + O(a^{n}).
\]
\end{prop}

\bigbreak
\subsection{Scattering phase estimate}
In this subsection we will develop the proof of Proposition~\ref{sigma.bound}.
If we let $\Omega := \supp V \subset \bbH^{n+1}$, the resolvent identity \eqref{r0rv} implies the relation
\begin{equation}
\left(1 - VR_V(s) \chr_\Omega\right)\left(1+VR_0(s) \chr_\Omega\right) = 1. 
\end{equation}
Proposition~\ref{res.est} implies that $\norm{VR_0(s) \chr_\Omega} <1$ for $|s-\nh|$ sufficiently large,
in which case we can write
\begin{equation}\label{VR.relation}
1 - VR_V(s) \chr_\Omega = \left(1+VR_0(s) \chr_\Omega\right)^{-1}.
\end{equation}

\begin{lemma}\label{relS.lemma}
The scattering matrices satisfy a relative scattering formula
$$
S_V(s) S_0(s)^{-1} = 1 + (2s-n) E_0(s)^t \chr_\Omega \left( 1 + VR_0(s)\chr_\Omega \right)^{-1} V E_0(n-s),
$$
valid for $\re s \ge n/2$ with $|s-n/2|$ sufficiently large.
\end{lemma}

\begin{proof}
Using equation \eqref{r0rv} and its transpose we have
\[
\begin{split}
R_V(s) & = R_0(s) - R_0(s)VR_V(s) \\
& = R_0(s) - R_0(s) V R_0(s) + R_0(s)VR_V(s)VR_0(s) \\
& = R_0(s) - R_0 (s)\chr_\Omega \left( 1 + VR_V(s) \chr_\Omega \right) V R_0(s)
\end{split}
\]
The formulas for the scattering matrix can then be derived by multiplying
the kernels by $(2s-n) (\rho \rho')^{-s}$ and taking the limit as $\rho, \rho' \to 0$.
This gives
$$
S_V(s) = S_0(s) - (2s-n) E_0(s)^t \chr_\Omega \left( 1 - VR_V \chr_\Omega \right) V E_0(s).
$$
The result follows after applying $S_0(s)^{-1}$ on the right and using \eqref{VR.relation}.
\end{proof}

In order to apply Lemma~\ref{r0rv} we need some estimates on Hilbert-Schmidt norms of the Poisson operator.  For this estimate
it is easiest to write the Poisson kernel in the $\bbB^{n+1}$ model.  Recall that we use the boundary
defining function $\rho = 2e^{-r}$, where $r$ is hyperbolic distance from the origin.  The normalizing
factor is included so that the induced metric on $\del \bbB^{n+1} = S^n$ is the standard sphere metric.
For this boundary defining function, the Poisson kernel is given by
$$
E_0(s;u,\omega) = 2^{-s-1} \pi^{-n/2} \frac{\Gamma(s)}{\Gamma(s-\frac{n}2+1)} \left( \frac{1 - \abs{u}^2}{\abs{u - \omega}^2}
\right)^s,
$$
where $u \in \bbB^{n+1}$, $\omega \in S^n$.

\begin{lemma}\label{E0.hsnorm}
Let $\chi \in L^\infty_{\rm cpt}(\bbB^{n+1})$.
For $t \in \bbR$, the Poisson operator $E_0(\nh+it) : L^2(S^n) \to L^2(\bbB^{n+1})$ satisfies
$$
\norm{\chi E_0(\nh+it)}_2 \le C \abs{t}^{\frac{n}2-1},
$$ 
and
$$
\norm{\chi E_0'(\nh+it)}_2 \le C \abs{t}^{\frac{n}2-1}.
$$
\end{lemma}

\begin{proof}
The Hilbert-Schmidt norm is calculated directly:
$$
\norm{\chi E_0(\nh+it)}_2  = c_n \abs{\frac{\Gamma(\frac{n}2+it)}{\Gamma(1+it)}} \left[\int_{S^n} \int_{\bbB^{n+1}} 
\chi(u)^2 \left( \frac{1 - \abs{u}^2}{\abs{u - \omega}^2} \right)^n \>\>dV(u)\> d\omega \right]^{\frac12}
$$
Because $\chi$ is compactly supported, there is no convergence issue and
the term in brackets is just a constant.  The result follows from 
$$
\abs{\frac{\Gamma(\frac{n}2+it)}{\Gamma(1+it)}} \le C \abs{t}^{\frac{n}2-1},
$$
which is easily deduced from Stirling's formula.  The derivative estimate is similar.
\end{proof}

\bigbreak
\begin{proof}[Proof of Proposition~3.1]
By virtue of Lemma~\ref{relS.lemma} we can write this as
$$
\tau(s) := \det (1+ T(s)),
$$
where 
$$
T(s) := (2s-n) E_0(s)^t \chr_\Omega \left( 1 + VR_0(s)\chr_\Omega \right)^{-1} V E_0(n-s).
$$

Following the argument from Froese \cite[Lemma~3.3]{Froese:1998}, we will estimate the derivative
\begin{equation}\label{logtau.deriv}
\frac{\tau'}{\tau}(s) = \tr \Bigl[ (1+T(n-s)) T'(s) \Bigr].
\end{equation}
For $V$ real, $S_V(s) S_0(s)^{-1}$ is unitary for $\re s = \nh$, so that in this case $\norm{1+T(n-s)}=1$.
For a complex potential we need a separate estimate.  Note that for $\re s = \nh$,
\[
[\chr_\Omega E_0(s)]^* = E_0(n-s)^t \chr_\Omega,
\]
and we have the general relation (for any $s$),
\[
R_0(s) - R_0(n-s) = (n-2s) E_0(s) E_0(n-s)^t.
\]
Thus we can estimate, for $\re s = \nh$ with $\abs{s-\nh}$ sufficiently large,
\[
\norm{\chr_\Omega E_0(s)}^2 = \frac{1}{\abs{n-2s}} \norm{\chr_\Omega \bigl(R_0(s) - R_0(n-s)\bigr) \chr_\Omega}.
\]

The model resolvent estimate \eqref{R0.est} from Guillarmou \cite[Prop.~3.2]{Gui:2005c} thus 
implies that
\[
\norm{\chr_\Omega E_0(\nh + it)} = O(\abs{t}^{-1}),
\]
for $\abs{t}$ large, and also that
\begin{equation}\label{vr.inv}
\norm{\left( 1 + VR_0(\nh+it)\chr_\Omega \right)^{-1}} = O(1),
\end{equation}
for $\abs{t}$ sufficiently large.  We conclude that
\begin{equation}\label{onet.est}
\norm{1+T(\nh+it)} = O(1).
\end{equation}

By using \eqref{onet.est} with \eqref{logtau.deriv},
we can bound the derivative of the scattering phase by a trace norm,
$$
\abs{\frac{\tau'}{\tau}(\nh+it)} \le C \norm{T'(\nh+it)}_1.
$$
To control the trace norm, we have the Hilbert-Schmidt estimates on $E_0(n/2 \pm it)$
and derivatives from Lemma~\ref{E0.hsnorm}.  
Since Guillarmou \cite[Prop.~3.2]{Gui:2005c} proves that the estimate \eqref{R0.est} also holds with $R_0(s)$ 
replaced by the derivative $R_0'(s)$, we can estimate
$$
\norm{\del_t \left( 1 + VR_0(\nh+it)\chr_\Omega \right)^{-1}} = O(1),
$$
for $\abs{t}$ sufficiently large.
Putting these together (and noting the extra factor of $2s-n = 2it$) we obtain
$$
\norm{T'(\nh+it)}_1 = O(\brak{t}^{n-1}),
$$
and the result follows.
\end{proof}

\bigbreak
\subsection{General scattering determinant estimate}\label{scdet.sec}

Sharp upper bounds for $\abs{\tau(s)}$ were provided in Borthwick \cite[Prop.~5.4]{Borthwick:2010}, for a more general class of compactly supported ``black box'' perturbations of $\Delta$.  However, if we restrict to potential scattering we can improve the error estimate.  (In the Euclidean case this improvement was established by Dinh-Vu \cite{DV:2012}.)  
Theorem~\ref{upper.thm} follows immediately from the counting formula of 
Proposition~\ref{count.prop} and the following:
\begin{prop}\label{tau.upper}
Assume that the support of $V$ is contained within a ball of radius $r_0$.
For $\dist(ae^{i\theta},\bbN+\tfrac12)>a^{-\beta}$ for some $\beta>0$ and $\abs{\theta} \le \pi/2$ we have
$$
\log\abs{\tau(\nh+ae^{i\theta})} \le h_{r_0}(\theta) a^{n+1} + O(a^{n} \log a),
$$ 
uniformly for $\abs{\theta} \le \pi/2$, with $h_{r_0}(\theta)$ the indicator function defined in \eqref{indicator.def}.
\end{prop}
\begin{proof}
Set $r_j = r_0 + \tfrac{j}{a}$ for $j=1,2,3$.
Let $\psi\in\cinf(\bbR)$ be a cutoff function with $\psi(t)=1$ for $t\le 0$ and $\psi(t)=0$ for $t\ge 1$.  Then set
$\chi_j(r) = \psi(a(r-r_j))$, so that $\chi_j = 1$ for $r\le r_j$ and $\chi_j = 0$ for $r \ge r_{j+1}$.
Then from the proof of \cite[Lemma~4.1]{Borthwick:2010} we have
\[
S_V(s)S_0(s)^{-1} = 1+ (2s-n) E_0(s)^t [\Delta, \chi_2] R_V(s) [\Delta,\chi_1] E_0(n-s),
\]
where $E_0(s)$ is the unperturbed Poisson operator on $\bbH^{n+1}$.  
As in the proof of \cite[Lemma~5.2]{Borthwick:2010}, this formula leads to an estimate
\begin{equation}\label{tau.sum1}
\log \abs{\tau(s)} \le \sum_{l=0}^\infty \mu_n(l) \log (1 + A(s) \lambda_l(s)).
\end{equation}
Here $\mu_n(l)$ is the multiplicity of spherical harmonics of weight $l$ in dimension $n$,
\begin{equation}\label{hl.def}
\mu_n(l) := \frac{2l+n-1}{n-1} \binom{l+n-2}{n-2}.
\end{equation}
The factor $A(s)$ is a cutoff resolvent norm,
\[
A(s) := \bigl\Vert [\Delta_0, \chi_2] R_V(s) [\Delta_0,\chi_1] \bigr\Vert.
\]
Finally, the $\lambda_l$'s, calculated from explicit formulas for the Fourier decomposition $E_0(s)$, 
\[
\begin{split}
\lambda_l(s) &= \Bigl| \sin \pi (s-\nh) \> \Gamma(l + s) \Gamma(l+n-s) \Bigr| 
\left[ \int_{r_1}^{r_2} \Bigl|P_{\nu}^{-k}(\cosh r)\Bigr|^2\>\sinh r\>dr \right]^{\frac12} \\
&\qquad \times\left[ \int_{r_2}^{r_3} \Bigl|P_{\nu}^{-k}(\cosh r)\Bigr|^2\>\sinh r\>dr\right]^{\frac12},
\end{split}
\]
where $k := l + \tfrac{n-1}2$ and $\nu:= s-\frac{n+1}2$.

The key to the improved error in the potential scattering case is our ability to estimate $R_0(s)$ for 
$\re s \ge \nh$, as in Proposition~\ref{res.est}.  Using a rather general commutator argument (see, e.g., \cite[Lemma~9.8]{Borthwick}), we can derive from Proposition~\ref{res.est} the following bound:  if $\psi_1,\psi_2 \in \cinf(\bbH^{n+1})$ are cutoffs with disjoint supports, then for $a$ sufficiently large,
\[
\norm{\psi_1 R_0(\nh + ae^{i\theta}) \psi_2}_{\mathcal{L}(H^0,H^2)} \le Ca
\]
uniformly for $\abs{\theta} \le \tfrac{\pi}2$.   Then, noting that $\norm{\del_r^m \chi_j}_{\infty} = O(a^{m})$, we have
\begin{equation}\label{A.bd}
A(\nh + ae^{i\theta}) \le Ca^5.
\end{equation}

From \cite[Lemma~5.3]{Borthwick:2010} we also quote the estimate
\begin{equation}\label{lambda.bd}
\log \lambda_l(\nh+ae^{i\theta}) \le k H\left(\frac{ae^{i\theta}}{k}, r_3\right) + C_{r_0} \log k + C_\beta \log a,
\end{equation}
valid for $k>0$, $\abs{\theta} \le \tfrac{\pi}2$, 
and $\dist(ae^{i\theta},\bbN+\tfrac12)> a^{-\beta}$.  
The statement in \cite{Borthwick:2010} assumed that $a \in \bbN$, but that condition was just a simple way to avoid the poles of $\abs{\tan (\pi a e^{i\theta})}$.   We can easily extend the estimate by noting that $\abs{\tan \pi z} \le 1 + \dist(z,\bbZ+\tfrac12)^{-1}$.

Note that the case $k=0$ occurs only if $n=1$.  In this case, the complicated Legendre function bounds used in \cite{Borthwick:2010} can be replaced by a simpler estimate:
\[
P^{0}_{\nu -1/2}(\cosh r) = (2\pi \sinh r)^{-\frac12} (e^{\nu r} + i e^{-\nu r}) 
(1 + O(\brak{\nu}^{-1})),
\]
valid for fixed $r>0$ and $\arg \nu \in [0, \frac{\pi}2]$.

The remainder of the proof is similar to that of \cite[Lemma~5.3]{Borthwick:2010}, with some modifications to improve the error term.  By conjugation, it suffices to consider $\theta \in [0,\tfrac{\pi}2]$.  Let $x = \varrho(\theta)$ be the implicit solution of $H(xe^{i\theta}, r_0) = 0$, as illustrated in Figure~\ref{hplot.fig}.
\begin{figure}
\begin{center}
\begin{overpic}{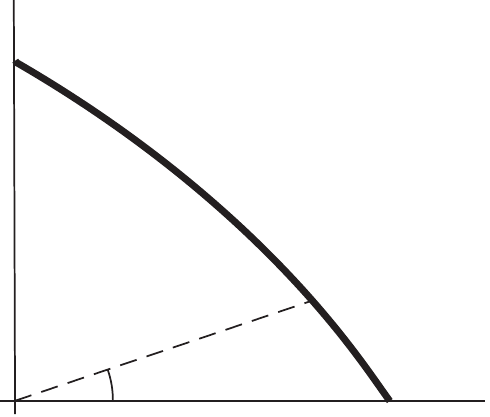}
\put(25,17){$\varrho(\theta)$}
\put(55,55){$H>0$}
\put(10,40){$H<0$}
\end{overpic}
\end{center}
\caption{Positive and negative regions for $H(\alpha,r)$, shown for $r = 1$.  The curve $\{H=0\}$ is parametrized in polar coordinates by $\varrho$.}\label{hplot.fig}
\end{figure}
We will use $\varrho(\theta)$ to subdivide the sum on the right side of 
\eqref{tau.sum1} into several pieces, the first of which is
\[
\Sigma_+ := \sum_{l:\>0< k \le a/\varrho(\theta)} \mu_n(l) \log (1 + A(s) \lambda_l(s)).
\]
Using the estimate,
\[
\mu_n(l) = \frac{2k^{n-1}}{\Gamma(n)} + O(k^{n-2}),
\]
with \eqref{A.bd} and \eqref{lambda.bd}, we have
\[
\Sigma_+ \le \frac{2}{\Gamma(n)} \sum_{0< k \le a/\varrho(\theta)} k^n H\left(\frac{ae^{i\theta}}{k}, r_3\right) + O(a^n \log a).
\]
Estimating the sum by an integral, and then substituting $x = a/k$, yields
\[
\Sigma_+ \le \frac{2a^{n+1}}{\Gamma(n)}  \int_{\varrho(\theta)}^\infty \frac{H(x e^{i\theta}, r_3)}{x^{n+2}}\>dx 
+ O(a^n \log a).
\]
By observing that 
\[
\del_r H(\alpha, r) = 2\re \left(\frac{\sqrt{1 + \alpha^2 \sinh^2 r}}{\sinh r} \right), 
\]
and recalling that $r_3 = r_0 + \frac{3}{a}$,
we see that $H(x e^{i\theta}, r_3) = H(x e^{i\theta}, r_0) + O(\tfrac{x}{a})$ for $x \ge \varrho(\theta)$, uniformly in $\theta$.
Hence
\begin{equation}\label{Splus}
\Sigma_+ \le h_{r_0}(\theta) a^{n+1} + O(a^n \log a),
\end{equation}

The second piece of \eqref{tau.sum1}, containing terms where $k \asymp a$, is defined as
\[
\Sigma_0 := \sum_{l:\>a/\varrho(\theta) < k \le 2a/\varrho(\theta)} \mu_n(l) \log (1 + A(s)) \lambda_l(s)).
\]
Note that for $k > a/\varrho(\theta)$,
\[
H\left(\frac{ae^{i\theta}}{k}, r_3\right) = O(a^{-1}).
\]
Thus, in the range of $\Sigma_0$, \eqref{lambda.bd} gives
\[
\log \lambda_l(\nh + ae^{i\theta}) = O(\log k).
\]
Since $\Sigma_0$ contains $O(a)$ terms, and $\mu_n(l) = O(k^{n-1})$, we conclude that
\begin{equation}\label{Szero}
\Sigma_0 = O(a^n \log a).
\end{equation}

Finally, the third part of  \eqref{tau.sum1} is
\[
\Sigma_- := \sum_{l:\>k > 2a/\varrho(\theta)} \mu_n(l) \log (1 + A(s) \lambda_l(s)).
\]
In this range we have $H\left( ae^{i\theta}/k, r_3\right) < -c$ for some $c>0$, when $a$ is sufficiently large.
It follows easily that
\[
\Sigma_- = O(e^{-ca}).
\]
\end{proof}

\begin{lemma}\label{lind.lemma}
Let $\mathcal{Q}$ denote the joint set of zeros and poles of $\tau(s)$.
Assuming $d(s, \mathcal{Q}) > \brak{s}^{-\beta}$ for some $\beta>2$, we have
$$
-c_\beta  \brak{s}^{n+1} \le \log |\tau(s)| \le C_\beta \brak{s}^{n+1}.
$$
\end{lemma}
\begin{proof}
Since $\tau(\nh-z) = \tau(\nh+z)^{-1}$, it suffices to prove the bounds for $\re z \ge 0$.
Proposition~\ref{tau.upper} gives the upper bound,
\begin{equation}\label{tauz.bd}
\log |\tau(\nh+z)| \le C_\beta\brak{z}^{n+1}, 
\end{equation}
for $\re z\ge 0$ with $\dist(z,\bbN-\tfrac{n+1}2) > \brak{z}^{-\beta}$.
Thus we have only to prove the lower bound.  

Consider the Hadamard products appearing in the factorization of
$\tau(s)$ given in Proposition~\ref{tau.factor}.   These products are of
order $n+1$ but not finite type.  To work around this, we consider products of the form $H_*(\nh + z)
H_*(\nh + e^{\pm i\pi/(n+1)} z)$.  By Lindel\"of's Theorem (see 
e.g.~\cite[Thm.~2.10.1]{Boas}), such functions are of finite type and so satisfy estimates,
\[
\log \abs{H_*\left(\nh + z\right) H_*\left(\nh + e^{\pm i\pi/(n+1)} z\right)} \le C \brak{z}^{n+1}.
\]
The Minimum Modulus Theorem \cite[Thm.~3.7.4]{Boas}, gives corresponding lower bounds,
\[
\log \abs{H_*\left(\nh + z\right) H_*\left(\nh + e^{\pm i\pi/(n+1)} z\right)} \ge -c_\beta \brak{z}^{n+1},
\]
provided we stay a distance at least $\brak{z}^{-\beta}$ away from the zeros, for some $\beta>2$.
Using these estimates together with Proposition~\ref{tau.factor} gives
\begin{equation}\label{tt.bnd}
\log \abs{\tau(\nh + z)}  \ge  - c_\beta \brak{z}^{n+1} -  \log\abs{\tau\left(\nh + e^{\pm i\pi/(n+1)}z\right)},
\end{equation}
provided $\nh +z$ and $\nh + e^{\pm i\pi/(n+1)}z$ stay at least a distance $\brak{z}^{-\beta}$
away from the sets $1- \calR_{F_{\ell,r_0}}$ and $\calR_{F_{\ell}}$.

Assuming $\arg z \in [-\tfrac{\pi}2 + \tfrac{\pi}{n+1}, \tfrac{\pi}2]$, 
we already know $\log \abs{\tau(\nh + e^{-i\pi/(n+1)} z)} \le C \brak{z}^{n+1}$ from \eqref{tauz.bd}, provided 
$e^{- i\pi/(n+1)} z$ stays at least a distance $\brak{z}^{-\beta}$ away from $\bbN-\tfrac{n+1}2$.
Similarly, for $\arg z \in [-\tfrac{\pi}2, \tfrac{\pi}2 - \tfrac{\pi}{n+1}]$, we already have an upper bound on $\log \abs{\tau(\nh + e^{i\pi/(n+1)} z)}$
In conjunction with (\ref{tt.bnd}), these estimates give the desired lower bound of $\log \abs{\tau(\nh + z)}$ in the first quadrant, except that we have been required to excise disks near the points not only of $\mathcal{Q}$, but also its rotations by $\pm \tfrac{\pi}{n+1}$.  However,
once we have obtained estimates of $\tau(\nh + z)$ itself, the missing disks can be filled in using the maximum modulus theorem.
\end{proof}

\bigbreak
\section{Scattering matrix elements for radial potentials}\label{scmatrix.sec}

For a radial potential, $V = V(r)$, the scattering matrix $S_V(s)$ acting on $S^n$ is diagonalized by spherical harmonics.  In this section we will develop a formula for the matrix elements of $S_V(s)$ which will then be used to produce estimates in \S\ref{asym.sec}.

In geodesic polar coordinates, $\bbH^n \cong \bbR_+ \times S^n$, and the Laplacian takes the form,
\[
\Delta = -\del_r^2 - n\coth r\>\del_r + \frac{1}{\sinh^2 r} \Delta_{S^n}.
\]
The spherical harmonic $Y_l^m$ is an eigenfunction of $\Delta_{S^n}$ satisfying
\[
\Delta_{S^n} Y_l^m = l(l+n-1)Y_l^m.
\]
The indices range over $l = 0, 1, 2, \dots$ and $m = 0, 1, \dots, \mu_n(l)$ with
\begin{equation}\label{hnl.def}
\mu_n(l) := \frac{2l+n-1}{n-1} \binom{l+n-2}{n-2}.
\end{equation}
As in \S\ref{potsc.sec}, we use the boundary defining function
\begin{equation}\label{rhor}
\rho = 2e^{-r}.
\end{equation}
This choice is made so that the metric induced on the conformal infinity by $\rho^2 g_{\bbH^{n+1}}$ is the standard sphere metric on $S^n$.

The scattering matrix elements $[S_V(s)]_l$ are the eigenvalues, meromorphic in $s$, of the spherical harmonics
\[
S_V(s) Y_l^m = [S_V(s)]_l Y_l^m.
\]
To compute $[S_V(s)]_l$, we consider a generalized eigenfunction $\phi(r,\theta) = u(r) Y_l^m(\theta)$ on $\bbH^{n+1}$.  From the eigenfunction equation,
\[
(\Delta - s(n-s)) \phi = 0,
\]
we derive the coefficient equation
\begin{equation}\label{coeff.eq}
\left[ -\del_r^2 - n \coth r\>\del_r + \frac{l(l+n-1)}{\sinh^2 r} - s(n-s)+ V(r) \right] u = 0.
\end{equation}
If we rewrite this equation in the variable $\rho$, then indicial roots at $\rho= 0$ are $s$ and $n-s$, implying that $u$ will in general have a two-part asymptotic expansion with leading terms of these orders as $\rho \to 0$.   The scattering matrix eigenvalue associated to $Y_l^m$ appears as the ratio of the leading coefficients, i.e.
\begin{equation}\label{Sl.def}
u \sim c_s \left( \rho^{n-s} + [S_V(s)]_l \rho^{s} \right),
\end{equation}
as $\rho \to 0$.  

The homogeneous equation ($V=0$) is solved by Legendre functions, with the independent solutions,
\begin{equation}\label{homog.soln}
\begin{split}
u_0^k(s;r) &:= (\sinh r)^{-\frac{n-1}2} P^{-k}_{\nu}(\cosh r), \\
v_0^k(s;r) &:= (\sinh r)^{-\frac{n-1}2} \bQ^k_{\nu}(\cosh r),
\end{split}
\end{equation}
where
$$
k := l + \frac{n-1}2, \qquad \nu := s - \frac{n+1}2.
$$
Here $\bQ^k_{\nu}$ is the normalized Q-function introduced by Olver \cite{Olver}, which is analytic in its parameters.  It is related to the standard definition by $Q^k_\nu = e^{i \pi k} \Gamma(k+\mu+1)\bQ^k_{\nu}$.
From the well-known asymptotics of the Legendre functions we obtain
\begin{equation}\label{homog.0asym}
\begin{split}
u_0^k(s;r) &\sim \frac{2^{-k}}{\Gamma(k+1)} r^l\quad\text{as }r\to 0, \\
v_0^k(s;r) &\sim \frac{2^{k-1}\Gamma(k)}{\Gamma(k+\nu+1)} r^{-l-n+1}\quad\text{as }r\to 0,
\end{split}
\end{equation}
and also
\begin{equation}\label{v.asym}
v_0^k(s;r) \sim \frac{\pi^\frac12}{2^{\nu+1} \Gamma(\nu+\frac{3}2)} \rho^s\quad\text{as }r\to \infty.
\end{equation}

If $V$ is assumed to have support in $\{r\le r_0\}$, there exists a solution $v^k(s;r)$ of the full equation \eqref{coeff.eq} that is equal to 
$v_0^k(s;r)$ for $r\ge r_0$.  This solution will generally have a leading singularity proportional to $r^{-l-n+1}$ at $r=0$, so that $v^k(s;r) Y_l^m(\theta)$ is not a smooth solution at the origin.  However, since $v^k(n-s)$ is an independent solution, we can cancel the singularity at $r=0$ by taking the combination,
\[
u^k(s;r) = F^k(n-s) v^k(s;r) - F^k(s) v^k(n-s;r),
\]
where the coefficients are given by the limits
\begin{equation}\label{Fk.def}
F^k(s) := \lim_{r\to 0} r^{l+n-1} v^k(s;r).
\end{equation}

By the indicial equation at $r=0$, canceling the leading $r^{-l-n+1}$ term at $r\to 0$ removes that whole part of the expansion, 
so that expansion of $u^k(s;r)$ at $r=0$ has only the part with leading term proportional to $r^l$.  
Hence $u^k(s;r)Y_l^m(\theta)$ is regular at the origin and defines a generalized eigenfunction on $\bbH^{n+1}$. 
We can therefore deduce from \eqref{Sl.def} and \eqref{v.asym} that
\begin{equation}\label{SV.AA}
[S_V(s)]_l = 2^{n-2s} \frac{\Gamma(\frac{n}2 - s)}{\Gamma(s-\frac{n}2)} \frac{F^k(n-s)}{F^k(s)},
\end{equation}
with $k :=l + (n-1)/2$.  

For future reference, we also introduce the unperturbed version of these coefficients,
\[
\begin{split}
F^k_0(s) & := \lim_{r\to 0} r^{l+n-1} v^k_0(s;r) \\
& =  \frac{2^{k-1}\Gamma(k)}{\Gamma(k+s-\frac{n-1}2)}.
\end{split}
\]
In this notation, the (well-known) formula for the unperturbed scattering matrix elements becomes
\begin{equation}\label{S0.AA}
[S_0(s)]_l = 2^{n-2s} \frac{\Gamma(\frac{n}2 - s)}{\Gamma(s-\frac{n}2)} \frac{F^k_0(n-s)}{F^k_0(s)}.
\end{equation}

\medbreak
For a radial step potential of the form $V = c\chi_{B_{(r_0)}}$, with $c\in\bbC$, we can write the functions $v^k(s;r)$ explicitly.  The coefficient solutions for $r \le r_0$ are Legendre functions $P^{-k}_{\omega(s)}(r_0)$, with the shifted parameter
$$
\omega(s) := -\tfrac12 + \sqrt{(s-\nh)^2+c}.
$$
A simple ODE matching problem at $r = r_0$ then shows that for $V = c\chi_{B_{(r_0)}}$, the coefficients $F^k(s)$ appearing in \eqref{SV.AA} are given by
\[
F^k_{c,r_0}(s) :=  \mathcal{W}\left[{\bQ^{k}_{s-\frac{n+1}2}}(z),  P^{-k}_{\omega(s)}(z)\right]\Big|_{z = \cosh r_0} 
\]
where $\mathcal{W}$ denotes the Wronskian.  The resonances in the $Y_l^m$ mode can then be characterized explicitly a the zeros of $F^k_{c,r_0}(s)$;  this is the basis of the resonance plots in Figures~\ref{RPotRes3.fig} and \ref{plot1I.fig}.

\bigbreak
\section{Radial matrix element asymptotics}\label{asym.sec}

The main goal of this section is a precise estimate of the eigenvalues of the relative scattering matrix $S_V(s) S_0(s)^{-1}$.  
For convenience, let us set
\begin{equation}\label{Lamk.def}
\Lambda_k(s) := \left[ S_V(s) S_0(s)^{-1} \right]_l,
\end{equation}
where $k := l + (n-1)/2$ as before.
\begin{prop}\label{lamk.prop}
For all $k \ge 0$, $\abs{\theta} < \tfrac{\pi}2-\vep$, and $a$ sufficiently large, and assuming that $\dist(ae^{i\theta},\bbZ/2) >\delta$, we have
\begin{equation}\label{lamk.asymp}
\abs{\Lambda_k(\nh+ae^{i\theta})}  \asymp (k^2+a^2)^{-\frac{\sigma+1}2} e^{k H(k^{-1}ae^{i\theta}; r_0)} + O(1),
\end{equation}
with constants that depend only on the potential $V$ and on $c,\vep,\delta$.  
(For $k=0$, the exponent $k H(k^{-1}ae^{i\theta}; r_0)$ is replaced by its limiting value, $r_0a \cos\theta$.)
\end{prop}

The strategy for the proof is analogous to that of Zworski \cite{Zworski:1989}.  Using \eqref{SV.AA} and \eqref{S0.AA} we can write
\begin{equation}\label{L.FF}
\Lambda_k(s) = \frac{F^k(n-s)}{F^k(s)} \frac{F^k_0(s)}{F^k_0(n-s)},
\end{equation}
By a standard application of variation of parameters to the ODE for $v^k(s)$, with $v^k(s;r) = v^k_0(s;r)$ for $r\ge r_0$ as the boundary condition, we obtain the integral equation 
\begin{equation}\label{v.recur}
v^k(s;r) = v^k_0(s;r) + \int_r^{r_0} J^k(s;r,t) V(t) v^k(s;t)\>dt,
\end{equation}
where the integral kernel is 
\[
J^k(s;r,t) := \frac{u^k_0(s;r) v^k_0(s;t) - u^k_0(s;t) v^k_0(s;r)}{\calW[u^k_0(s), v^k_0(s)](t)},
\]
with $\calW$ denoting the Wronskian.  Using the well-known formula for the Wronskian of a pair of Legendre functions, this kernel reduces to
\begin{equation}\label{J.def}
\begin{split}
J^k(s;r,t) &= \Gamma(k+\nu+1) (\sinh r)^{-\frac{n-1}2}  (\sinh t)^{\frac{n+1}2} \\
&\qquad\times \Bigl[ P^{-k}_{\nu}(\cosh r)  \bQ^k_{\nu}(\cosh t) - P^{-k}_{\nu}(\cosh t) \bQ^k_{\nu}(\cosh r)\Bigr].
\end{split}
\end{equation}

Formally, we can solve the integral equation for $v^k$ using the series $\sum_{j=0}^\infty v^k_j$, with $v^k_0$ the unperturbed solution and $v^k_j$ defined recursively by
\begin{equation}\label{vj.recur}
v^k_{j+1}(s;r) := \int_r^{r_0} J^k(s;r,t) V(t) v^k_j(s;t)\>dt.
\end{equation}
We first need asymptotic estimates on these $v^k_j$ which justify the convergence of this series, for $k$ sufficiently large.
Then we can derive estimates for the scattering matrix elements using \eqref{L.FF}.

\bigbreak
\subsection{Legendre function asymptotics}\label{leg.asym.sec}
To analyze the sequence $\{v^k_j\}$ we first recall some asymptotic estimates on the Legendre functions from \cite{Borthwick:2010},
obtained using techniques from Olver \cite{Olver}.  Set $\alpha = (s-\nh)/k$, so that $\nu = -\tfrac12 + k\alpha$.  
The Liouville transformation takes the Legendre equation to an approximate Airy equation with the variable $\zeta$ defined by
\begin{equation}\label{zeta.def}
\tfrac23 \zeta^{\frac32} = \phi,
\end{equation}
where
\begin{equation}\label{phi.def}
\phi(\alpha, r) := 
\alpha \log \left( \frac{\alpha \cosh r +  \sqrt{1 + \alpha^2 \sinh^2 r}}{\sqrt{\alpha^2-1}} \right) 
+ \frac12 \log \left[ 
\frac{\cosh r - \sqrt{1 + \alpha^2 \sinh^2 r}}{\cosh r + \sqrt{1 + \alpha^2 \sinh^2 r}} \right].
\end{equation}
The Legendre functions can then be approximated in terms of Airy functions of $\zeta$.

The asymptotics of $\phi(\alpha, \cdot)$ can be worked out fairly easily.  As $r\to 0$, we have
\begin{equation}\label{phi.asym1}
\phi(\alpha, r) = \log \left(\frac{r}2\right) + p(\alpha) + O(r^2),
\end{equation}
where 
\begin{equation}\label{pal.def}
p(\alpha) := \frac{\alpha}2 \log \left(\frac{\alpha+1}{\alpha-1}\right)
+ \frac12 \log (1-\alpha^2).
\end{equation}
And as $r \to \infty$, we have
\begin{equation}\label{phi.asym2}
\phi(\alpha, r) = \alpha r + q(\alpha) + O(r^{-2}),
\end{equation}
where 
\begin{equation}\label{qal.def}
q(\alpha) := \alpha \log \left(\frac{\alpha}{\sqrt{\alpha^2-1}}\right)
+ \frac12 \log \left(\frac{1-\alpha}{1+\alpha}\right),
\end{equation}

The Liouville transformation yields the following asymptotic result, derived in Borthwick \cite[Prop~A.1]{Borthwick:2010}.
Assuming that $k>0$, $\arg \alpha \in [0,\tfrac{\pi}2]$ and $r \in [0, \infty)$, we have 
\begin{equation}\label{p.asym}
P_{-\frac12+k\alpha}^{-k}(\cosh r) =  \frac{2\pi^{\frac12}}{\Gamma(k+1)}
\frac{k^{\frac16} \zeta^{\frac14} e^{\frac{\pi i}6}}{\bigl[1+\alpha^2 \sinh^2 r\bigr]^{\frac14}} 
\>e^{-k p(\alpha)}   \Bigl[ \Ai\bigl(k^{\frac23} e^{\frac{2\pi i}3} \zeta\bigr)
+ h_1(k, \alpha, r)\Bigr],
\end{equation}
and
\begin{equation}\label{q.asym}
\bQ_{-\frac12+k\alpha}^{k}(\cosh r) =  
\frac{2\pi}{\Gamma(k\alpha+1)} \frac{k^{\frac16} \zeta^{\frac14} 
(\frac{\alpha}2)^{\frac12}}{\bigl[1+\alpha^2 \sinh^2 r\bigr]^{\frac14}} 
\> e^{kq(\alpha)} \Bigl[ \Ai\bigl(k^{\frac23} \zeta\bigr) + h_0(k,\alpha, r)\Bigr],
\end{equation}
with the error estimates,
\begin{equation}\label{hestimates}
\begin{split}
|k^{\frac16} \zeta^{\frac14} h_1(k,\alpha,r)| & \le Ce^{k\re \phi} k^{-1} \Bigl(1+ |\alpha|^{-\frac23}\Bigr),\\
|k^{\frac16} \zeta^{\frac14} h_0(k,\alpha,r)| & \le Ce^{-k\re \phi} k^{-1} \Bigl(1+ |\alpha|^{-\frac23}\Bigr).
\end{split}
\end{equation}

For the most part, we will be content with the sharp upper bounds derived from these asymptotics.  From \cite[Cor.~A.3]{Borthwick:2010} we cite:
\begin{prop}\label{PQ.bounds}
Assuming that $|k\alpha| \ge 1$, $\arg \alpha \in [0,\tfrac{\pi}2-\vep]$, and $r \in [0, r_0]$, we have the following estimates: 
\begin{equation}\label{p.upper}
\Bigl| P_{-\frac12+k\alpha}^{-k}(\cosh r) \Bigr| \le  \frac{C}{\Gamma(k+1)} 
\> e^{k \re [\phi(\alpha,r) - p(\alpha)]}
\end{equation}
and
\begin{equation}\label{q.upper}
\Bigl| \bQ_{-\frac12+k\alpha}^{k}(\cosh r) \Bigr| \le  
\frac{C |\alpha|^{\frac12}}{|\Gamma(k\alpha+1)|} \> e^{-k\re[\phi(\alpha, r) - q(\alpha)]},
\end{equation}
where $C$ depends only on $r_0$ and $\vep$.  
\end{prop}

\medbreak
Beyond the upper bounds of Proposition~\ref{PQ.bounds}, which serve to control the error terms in our expansion, we need also a lower bound to apply to the leading term.

\begin{prop}\label{P.asym.prop}
Assume that $\arg \alpha \in [0,\tfrac{\pi}2-\vep]$, $r \in [0, r_0]$, and that for some sufficiently large $N$ we have both $k\ge N$ and $\abs{k\alpha} \ge N$.   Then
\[
P_{-\frac12+k\alpha}^{-k}(\cosh r) \asymp \frac{e^{-kp(\alpha)} }{\Gamma(k+1) \bigl[1+\alpha^2 \sinh^2 r\bigr]^{\frac14}} 
\>e^{k\phi(\alpha;r)}.
\]
\end{prop}
\begin{proof}
The assumption that $\arg \alpha \in [0,\tfrac{\pi}2-\vep]$ implies that $\arg \phi(\alpha;r) \in [0,\tfrac{\pi}2-\vep]$ also. Hence 
$\arg\zeta$ is bounded away from $\pi$ and we may apply the Airy function asymptotic \cite[eq.~(4.4.03)]{Olver},
\begin{equation}\label{ai.asym1}
\Ai(w) = \frac{1}{2\pi^{\frac12}} w^{-\frac14} \exp\bigl(-\tfrac23 w^{\frac32}\bigr)\bigl[1 + O_\epsilon(|w|^{-\frac32})\bigr],
\end{equation}
for $\abs{\arg w} \le \pi - \epsilon$.
This result, along with \eqref{p.asym} and \eqref{hestimates}, gives
\[
\begin{split}
P_{-\frac12+k\alpha}^{-k}(\cosh r) & =
\frac{e^{-kp(\alpha)} }{\Gamma(k+1) \bigl[1+\alpha^2 \sinh^2 r\bigr]^{\frac14}} 
\>e^{k\phi(\alpha;r)} \\
&\qquad \times \left[ 1 + O\left(k^{-1}\right) + O\left(k^{-\frac13} \abs{k\alpha}^{-\frac23}\right) 
+ O_\vep(\abs{k\phi}^{-1})\right]
\end{split}
\]
To complete the proof, we can deduce from the analysis of $\phi$ in the proof of \cite[Prop~A.1]{Borthwick:2010} that,
for $\arg \alpha \in [0,\tfrac{\pi}2-\vep]$,
\[
\abs{\phi(\alpha; r)} \ge c_\vep \min(\abs{\alpha},1).
\]
\end{proof}

\bigbreak
\subsection{Integral estimates}

For the application of Proposition~\ref{PQ.bounds} to the estimate of the iterated solutions $v^k_j$, we essentially need only two estimates for the inductive step, corresponding to the two exponentials appearing in the asymptotics.   For the second of these estimates we will need to bring in the hypothesis of Theorem~\ref{asymp.thm}: that $V(r)$ is continuous near $r=r_0$ and satisfies
\begin{equation}\label{V.asym}
V(r) \sim \kappa(r-r_0)^{\sigma-1} \quad\text{as }r\to r_0,
\end{equation} 
for some constants $\kappa \ne 0$ and $\sigma\ge 1$.

\begin{lemma}\label{Jekp.lemma}
For $k\ge 1$, $\abs{s-\nh} \ge1$ and $\abs{\arg (s-\nh)} < \tfrac{\pi}2 - \vep$, we have
\begin{equation}\label{Jekp.bnd}
\abs{\int_r^{r_0} J^k(s; r, t) V(t) e^{-k\re\phi(\alpha;r)} (\sinh r)^{-\frac{n-1}2} \>dt}
\le \frac{C}{k} e^{-k\re\phi(\alpha;r)} (\sinh r)^{-\frac{n-1}2},
\end{equation}
where $\alpha = (s - \nh)/k$.  Under the same hypotheses, and assuming also \eqref{V.asym}, there exists $N_\vep$ such that for $k \ge N_\vep$,
\begin{equation}\label{Jnekp.bnd}
\begin{split}
&\abs{\int_r^{r_0} J^k(s; r, t) V(t) e^{k\re\phi(\alpha;r)} (\sinh r)^{-\frac{n-1}2} \>dt} \\
&\hskip.5in \le \frac{C}{k} \left[e^{k\re\phi(\alpha;r)} + k^{-\sigma}e^{k\re(2\phi(\alpha;r_0)-\phi(\alpha;r))} \right] (\sinh r)^{-\frac{n-1}2}.
\end{split}
\end{equation}
\end{lemma}
\begin{proof}
The assumptions on $s$ correspond to the hypotheses of Proposition~\ref{PQ.bounds}.  By the conjugation symmetry we can assume $\arg \alpha \in [0, \tfrac{\pi}2 - \vep)$.

Applying the estimates \eqref{p.upper} and \eqref{q.upper} in the definition of $J(s;r,t)$ from \eqref{J.def} gives
\begin{equation}\label{ekp.int}
\begin{split}
&\abs{\int_r^{r_0} J^k(s; r, t) V(t) e^{-k\re\phi(\alpha;r)} (\sinh r)^{-\frac{n-1}2} \>dt}  \\
&\hskip.5in \le C_\vep  \abs{\frac{\alpha^{\frac12} \Gamma(k\alpha+k+\tfrac12)}{\Gamma(k+1)\Gamma(k\alpha+1)}}
e^{k\re[q(\alpha)-p(\alpha)]}\\
&\hskip.5in\qquad\times
\int_r^{r_0} \left[ e^{k\re(\phi(\alpha;r) - 2\phi(\alpha;t))} +  e^{-k\re \phi(\alpha;r)} \right] \abs{V(t)} \sinh t\>dt. 
\end{split}
\end{equation}
Since $\re\phi(\alpha,\cdot)$ is increasing, we can replace the expression in brackets by $2 e^{-k\re\phi(\alpha;r)}$.
The claim \eqref{Jekp.bnd} then follows easily from an estimate based on Stirling's formula:
\begin{equation}\label{ggg.est}
\abs{\frac{\alpha^{\frac12} \Gamma(k\alpha+k+\tfrac12)}{\Gamma(k+1)\Gamma(k\alpha+1)}} e^{k\re[q(\alpha)-p(\alpha)]}
\le Ck^{-1},
\end{equation}
valid for $\re\alpha\ge 0$ and $k>0$.

For the estimate \eqref{Jnekp.bnd}, the analog of \eqref{ekp.int}, together with \eqref{ggg.est}, yields
\[
\begin{split}
&\abs{\int_r^{r_0} J^k(s; r, t) V(t) e^{-k\re\phi(\alpha;r)} (\sinh r)^{-\frac{n-1}2} \>dt} \\
&\hskip.5in \le C_\vep k^{-1} \int_r^{r_0} \left[ e^{k\re \phi(\alpha;r)} +  e^{-k\re (\phi(\alpha;r) + 2\phi(\alpha;t))} \right] \abs{V(t)} \sinh t\>dt. 
\end{split}
\]
The bound on the first term in the bracket is clear, since $\re\phi(\alpha,\cdot)$ is increasing.  For the second term, we need to apply Prop~\ref{Laplace.prop}
and the assumption \eqref{V.asym}.  To check the hypotheses, we simply compute the derivatives,
\begin{equation}\label{phi.deriv}
\phi'(\alpha; r) = \frac{\sqrt{1+\alpha^2 \sinh^2 r}}{\sinh r}, \qquad
\phi''(\alpha; r) = - \frac{\coth r}{\sqrt{1+\alpha^2 \sinh^2 r}}
\end{equation}
where the prime denotes an $r$ derivative.  In particular, for $\arg \alpha \in [0, \tfrac{\pi}2 - \vep)$ we have
$\re \phi'>0$ and
\[
\re \phi'(\alpha; r_0) \ge c_\vep, \qquad \abs{\phi''(\alpha, r)} \le C_\vep.
\]
Proposition~\ref{Laplace.prop} then gives
\[
\int_r^{r_0} e^{2k\re \phi(\alpha;t)} \abs{V(t)} \sinh t\>dt \le C_{\vep, V} k^{-\sigma} e^{2k\re\phi(\alpha; r_0)},
\]
for $k \ge N_\vep$.
\end{proof}

\bigbreak
\subsection{Recursive estimates}

In the estimates that follow, we consider expressions involving both $s$ and $n-s$, but always with the convention that $\re s-\nh \ge 0$ (equivalently, $\re\alpha\ge 0$).  Our recursive estimates are most straightforward in the case of $v^k_j(s)$, which corresponds to the expansion of the denominator in formula \eqref{SV.AA} for the scattering matrix element.  
 
\begin{lemma}\label{vj.good}
For $k\ge 1$, $\abs{k\alpha} \ge1$ and $\abs{\arg \alpha} < \tfrac{\pi}2 - \vep$, there exists a constant $A = A(\vep, V, r_0)$, such that
\begin{equation}\label{vj.bound}
\abs{v^k_{j}(\nh+k\alpha; r)} \le \left(\frac{A}{k}\right)^{j}
\frac{\abs{\alpha}^\frac12}{\abs{\Gamma(k\alpha+1)}}
e^{-k\re[\phi(\alpha;r) - q(\alpha)]} (\sinh r)^{-\frac{n-1}2}.
\end{equation}
\end{lemma}
\begin{proof}
The proof is by induction over $j$.  
The estimate for $v_0$ follows directly from Corollary~\ref{PQ.bounds}.  And the inductive step that extends the estimate
from $v^k_{j-1}$ to $v^k_{j}$ the follows immediately from the iterative formula \eqref{vj.recur} and Lemma~\ref{Jekp.lemma}.
\end{proof}

To handle the terms involving $n-s = \nh-k\alpha$, it is useful to note the Legendre function identities,
\begin{equation}\label{leg.refl}
\begin{split}
P^{-k}_{-1-\nu}(z) &= P^{-k}_\nu(z),  \\
\bQ^k_{-1-\nu}(z)  &= \Gamma(k+\nu+1)\cos (\nu\pi) P^{-k}_\nu(z) + \frac{\Gamma(k+\nu+1)}{\Gamma(k-\nu)} \bQ^k_{\nu}(z).
\end{split}
\end{equation}
These imply the symmetry
$$
J(n-s;r,t) = J(s;r,t).
$$
We can also write the $v_0(\nh-k\alpha)$ solution in terms of $P^{-k}_\nu(z)$ and $\bQ^k_{\nu}(z)$ using \eqref{leg.refl}.  
Then for $\abs{\arg\alpha} < \tfrac{\pi}2-\vep$, Proposition~\ref{PQ.bounds} gives the estimate
\begin{equation}\label{vk0.left}
v^k_0(\nh-k\alpha; r) \le C_\vep \left[ \abs{a_k(\alpha)} e^{k\re\phi(\alpha;r)} + \abs{b_k(\alpha)} e^{-k\re\phi(\alpha;r)} \right] (\sinh r)^{-\frac{n-1}2},
\end{equation}
where 
\[
\begin{split}
a_k(\alpha) &:= \frac{\Gamma(k+k\alpha+\frac12) \sin(\pi k\alpha)}{\Gamma(k+1)} e^{-kp(\alpha)},\\
b_k(\alpha) &:= \frac{\alpha^\frac12 \Gamma(k+k\alpha+\frac12)}{\Gamma(k\alpha+1) \Gamma(k-k\alpha+\frac12)} e^{kq(\alpha)}.
\end{split}
\]

\begin{lemma}\label{vj.bad}
For $k\ge 1$, $\abs{k\alpha} \ge1$ and $\abs{\arg \alpha} < \tfrac{\pi}2 - \vep$,
\[
\begin{split}
\abs{v^k_j(\nh-k\alpha;r)} & \le \left(\frac{C_{\vep, V}}{k}\right)^j
\biggl[ \abs{a_k(\alpha)} e^{k\re\phi(\alpha;r)} \\
&\hskip.5in + \left(j\abs{a_k(\alpha)} k^{-\sigma} e^{2k\re\phi(\alpha; r_0)}+ \abs{b_k(\alpha)} \right) 
e^{-k\re\phi(\alpha;r)}\biggr]  (\sinh r)^{-\frac{n-1}2}.
\end{split}
\]
\end{lemma}
\begin{proof}
The $j=0$ case was already dealt with in \eqref{vk0.left}.  And if we start from the inductive assumption that
\[
\abs{v^k_j(\nh-k\alpha;r)} \le \left[ A_j e^{k\re\phi(\alpha;r)} + B_j e^{-k\re\phi(\alpha;r)} \right] (\sinh r)^{-\frac{n-1}2},
\]
then under these assumptions Lemma~\ref{Jekp.lemma} implies the bound
\[
\begin{split}
&\abs{v^k_{j+1}(\nh-k\alpha;r)} \\
&\qquad \le \frac{C_{\vep, V}}{k} \left[ A_j e^{k\re\phi(\alpha;r)} + \left(A_j k^{-\sigma} e^{2k\re\phi(\alpha; r_0)}+ B_j\right) 
e^{-k\re\phi(\alpha;r)}\right]  (\sinh r)^{-\frac{n-1}2}.
\end{split}
\]
The result follows by induction.
\end{proof}

\bigbreak
\subsection{High-frequency asymptotics}\label{Hfreq.sec}

From Lemmas~\ref{vj.good} and \ref{vj.bad} we deduce that for $\abs{\arg \alpha} < \tfrac{\pi}2 - \vep$ and $\abs{k\alpha} \ge 1$ there exists $N_{\vep, V}$ such that the two series,
\[
v^k(\nh \pm k\alpha;r) = \sum_{j=0}^\infty v^k_j(\nh \pm k\alpha;r),
\]
converge absolutely and uniformly for $r\in[0,r_0]$ and $k\ge N_{\vep, V}$.
Under these assumptions, we can then express the scattering matrix element $\Lambda_k(s)$ given in \eqref{L.FF} as a sum over of the limiting values
\[
F^k_j(s) := \lim_{r\to 0} r^{l+n-1} v^k_j(s;r).
\]
To analyze $\Lambda_k(s)$, we will need to consider the ratios $F^k_j(s)/F^k_0(s)$.  The dominant term is given by the following:
\begin{lemma}\label{A1A0.lemma}
For $\abs{\arg \alpha} \le \tfrac{\pi}2 - \vep$,  $k \ge N_{\vep, V}$, and $\abs{k\alpha} \ge N_{\vep, V}$,
and $\dist(k\alpha, \bbZ/2) > \vep$,
\[
\abs{\frac{F^k_1(\nh - k\alpha)}{F^k_0(\nh - k\alpha)}} \asymp  (k\brak{\alpha})^{-1-\sigma} e^{k H(\alpha, r_0)} + O(k^{-1}),
\]
with constants that depend only on $\vep, V$.
\end{lemma}
\begin{proof}
The exact formula for $F^k_1(n-s)$ can easily be deduced from the formula for $v^k_1(n-s;r)$,
\[
F^k_1(n-s) = \lim_{r\to 0}r^{l+n-1}  \int_r^{r_0} J^k(s;r,t) V(t) v^k_0(n-s;t)\>dt.
\]
Using the asymptotics \eqref{homog.0asym} and the definition \eqref{J.def} of $J^k$ we find that
\[
F^k_1(\nh - k\alpha) = 2^{k-1} \Gamma(k) \int_0^{r_0} P^{-k}_{-\frac12-k\alpha}(\cosh t) \bQ^{k}_{-\frac12-k\alpha}(\cosh t) V(t)\sinh t\>dt.
\]
To apply our Legendre estimates, which require $\re \nu \ge -\tfrac12$, we use the reflection formulas \eqref{leg.refl} to flip the arguments from $\nu \to -1-\nu$.
The resulting expression for the ratio $F^k_1/F^k_0$ has two parts, corresponding to the two terms in the $Q$-reflection formula.  We will write these as:
\[
\frac{F^k_1(\nh - k\alpha)}{F^k_0(\nh - k\alpha)} = I_1 + I_2,
\]
where
\[
I_1 :=  \Gamma(k+k\alpha+\tfrac12) \Gamma(k-k\alpha+\tfrac12) \sin (\pi k\alpha)\int_0^{r_0} \left(P^{-k}_{-\frac12+k\alpha}(\cosh t)\right)^2 V(t)\sinh t\>dt,
\]
and
\[
I_2 := \Gamma(k+k\alpha+\tfrac12) \int_0^{r_0} P^{-k}_{-\frac12+k\alpha}(\cosh t) \bQ^{k}_{-\frac12+k\alpha}(\cosh t) V(t)\sinh t\>dt.
\]

As usual, by the conjugation symmetry we assume that $\arg \alpha \in [0,\tfrac{\pi}2-\vep]$.  By Proposition~\ref{P.asym.prop}, for $\alpha$ in this sector we can find $N_\vep$ such that for $k\ge N_\vep$ and $\abs{k\alpha} \ge N_\vep$ we have
\[
P^{-k}_{-\frac12+k\alpha}(\cosh t) \asymp [1+\alpha^2\sinh^2 t]^{-\frac14} \frac{e^{k\phi(\alpha;t) - kp(\alpha)}}{\Gamma(k+1)},
\]
with constants that depend only on $\vep$ and $r_0$.  Thus, 
\[
I_1 \asymp \frac{\Gamma(k-k\alpha+\frac12) \Gamma(k+k\alpha+\frac12) \sin(\pi k\alpha)}{\Gamma(k+1)^2} e^{-2p(\alpha)}
\int_0^{r_0} e^{2k\phi(\alpha;t)} V(t) \frac{\sinh t}{\sqrt{1+\alpha^2\sinh^2t}}\>dt.
\]

We can estimate the integral using the version of Laplace's method given in Prop~\ref{Laplace.prop}.
From the expressions \eqref{phi.deriv} for the derivatives of $\phi$, we see that for  $\arg \alpha \in [0,\tfrac{\pi}2-\vep]$,
\[
\abs{\phi'(\alpha;t)} \ge c_\vep,\quad \frac{\re\phi'(\alpha;r_0)}{\abs{\phi'(\alpha;r_0)}} \ge c_\vep,
\quad \sup_{[r_0-\vep,r_0]} \abs{\phi''(\alpha;\cdot)} \le C_\vep.
\]
Under the assumption \eqref{V.asym}, Proposition~\ref{Laplace.prop} then gives, for $k \ge N_{\vep, V}$, 
\[
\int_0^{r_0} e^{2k\re \phi(\alpha;t)} \abs{V(t)} \sinh t\>dt \asymp \frac{(k\re\phi'(\alpha;r_0))^{-\sigma} e^{2k\phi(\alpha;r_0)}}{\sqrt{1+\alpha^2\sinh^2r_0}},
\]
with constants that depend only on $\vep$ and $V$.  The formula for $\phi'$ was given in \eqref{phi.deriv}.
For  $\arg \alpha \in [0,\tfrac{\pi}2-\vep]$ the integral estimate reduces to
\[
\int_0^{r_0} e^{2k\re \phi(\alpha;t)} \abs{V(t)} \sinh t\>dt \asymp k^{-\sigma} \brak{\alpha}^{-1-\sigma} e^{2k\phi(\alpha;r_0)}.
\]

An application of Stirling's formula gives the estimate, for $\re \alpha \ge 0$ and $\dist(k\alpha, \bbZ/2) > \vep$,
\begin{equation}\label{Stir.H}
\begin{split}
&\log\abs{\frac{\Gamma(k-k\alpha+\frac12) \Gamma(k+k\alpha+\frac12) \sin(\pi k\alpha)}{\Gamma(k)\Gamma(k+1)}} \\
&\qquad = k \re\Bigl[(\alpha+1)\log (\alpha+1) -  (\alpha-1)\log (\alpha-1)  \Bigr] + O(\log (1+\vep^{-1})).
\end{split}
\end{equation}
We note also that by definition,
\begin{equation}\label{H.2pp}
H(\alpha, r) = \re \Bigl[ 2\phi(\alpha, r) - 2p(\alpha) + (\alpha+1) \log (\alpha+1) - (\alpha-1) \log (\alpha-1) \Bigr].
\end{equation}
The combined estimate, for $\abs{\arg \alpha} \le \tfrac{\pi}2 - \vep$,  $k\ge N_{\vep, V}$, $\abs{k\alpha} \ge N_{\vep, V}$, 
and $\dist(k\alpha, \bbZ/2) > \vep$, is
\[
\abs{I_1} \asymp (k\brak{\alpha})^{-1-\sigma} e^{k H(\alpha, r_0)},
\]
with constants that depend only on $\vep$ and $V$.

To control the second integral, we apply Corollary~\ref{PQ.bounds} to estimate
\[
\abs{I_2} \le \abs{\frac{\alpha^{\frac12} \Gamma(k\alpha+k+\tfrac12)}{\Gamma(k+1)\Gamma(k\alpha+1)}} e^{k\re[q(\alpha)-p(\alpha)]}
\norm{V}_\infty \sinh r_0,
\]
under the same assumptions as for $I_1$.  The Stirling estimate \eqref{ggg.est} then shows that $I_2 \le C_{\vep, V}k^{-1}$.
\end{proof}

\medbreak
The other terms in the series expansion for $\Lambda_k(s)$ can now be estimated using Lemmas~\ref{vj.good} and \ref{vj.bad}.
\begin{lemma}\label{Aj.rem}
For $\abs{\arg \alpha} < \tfrac{\pi}2 - \vep$, $\abs{k\alpha} \ge1$, and $k \ge N_{\vep, V}$ we have
\begin{equation}\label{FF.good}
\abs{\frac{F^k_j(\nh+k\alpha)}{F^k_0(\nh+k\alpha)}} \le \left(\frac{C_{\vep,V}}{k}\right)^{j},
\end{equation}
and
\begin{equation}\label{FF.bad}
\abs{\frac{F^k_j(\nh-k\alpha)}{F^k_0(\nh-k\alpha)}} \le \left(\frac{C_{\vep, V}}{k}\right)^j
\left[jk^{-\sigma} e^{kH(\alpha;r_0)} + 1\right].
\end{equation}
\end{lemma}
\begin{proof}
From Lemma~\ref{vj.good} we obtain the estimates, for $k\ge 1$, $\abs{k\alpha} \ge1$ and $\arg \alpha\in [0,\tfrac{\pi}2 - \vep)$,
\[
\abs{F^k_j(\nh+k\alpha)} \le \left(\frac{C_{\vep,V}}{k}\right)^{j}
\frac{2^k \abs{\alpha}^\frac12}{\abs{\Gamma(k\alpha+1)}} e^{k\re[q(\alpha) - p(\alpha)]}.
\]
In conjunction with the Stirling estimate \eqref{ggg.est} this gives \eqref{FF.good}.

For the estimate \eqref{FF.bad} in the other half-plane, we start by taking the $r\to0$ limit in Lemma~\ref{vj.bad} to obtain, 
for $k\ge 1$, $\abs{k\alpha} \ge1$ and $\abs{\arg \alpha} < \tfrac{\pi}2 - \vep$,
\[
\abs{F^k_j(\nh-k\alpha)}  \le \left(\frac{C_{\vep, V}}{k}\right)^j
\left(j\abs{a_k(\alpha)} k^{-\sigma} e^{2k\re\phi(\alpha; r_0)}+ \abs{b_k(\alpha)} \right) 
2^k e^{-k\re p(\alpha)}.
\]
After dividing by $F^k_0(\nh-k\alpha)$ and substituting the definitions of $a_k$ and $b_k$, we find that
\[
\begin{split}
\abs{\frac{F^k_j(\nh-k\alpha)}{F^k_0(\nh-k\alpha)}} & \le \left(\frac{C_{\vep, V}}{k}\right)^j
\biggl[\frac{2\Gamma(k-k\alpha+\frac12) \Gamma(k+k\alpha+\frac12) \sin(\pi k\alpha)}{\Gamma(k)\Gamma(k+1)} j k^{-\sigma} e^{2k\re(\phi(\alpha; r_0)-p(\alpha))} \\
&\hskip1in + \frac{2\alpha^\frac12 \Gamma(k+k\alpha+\frac12)}{\Gamma(k)\Gamma(k\alpha+1)} e^{k\re(q(\alpha)-p(\alpha))} \biggr].
\end{split}
\]
The first expression is estimated using \eqref{Stir.H} and \eqref{H.2pp} and the second by \eqref{ggg.est}.
\end{proof}

\bigbreak
\begin{proof}[Proof of Proposition~\ref{lamk.prop} (part one)]
By \eqref{L.FF} and Lemmas~\ref{vj.good} and \ref{vj.bad}, for $k \ge N_{\vep, V}$ we can represent
\[
\Lambda_k(\nh + k\alpha) = \left[1 + \sum_{j=1}^\infty \frac{F^k_j(\nh-k\alpha)}{F^k_0(\nh-k\alpha)} \right] 
\left[1 + \sum_{j=1}^\infty \frac{F^k_j(\nh+k\alpha)}{F^k_0(\nh+k\alpha)} \right]^{-1}.
\]
Under the hypotheses, \eqref{FF.good} gives
\[
\sum_{j=1}^\infty \abs{\frac{F^k_j(\nh+k\alpha)}{F^k_0(\nh+k\alpha)}} \le C_{\vep, V}k^{-1},
\]
while from \eqref{FF.bad} we have
\[
\sum_{j=2}^\infty \abs{\frac{F^k_j(\nh-k\alpha)}{F^k_0(\nh-k\alpha)}} \le C_{\vep, V} k^{-2} \left[k^{-\sigma} e^{kH(\alpha;r_0)} + 1\right].
\]

Assuming $k \ge N_{\vep, V}$, these estimates give
\[
\Lambda_k(\nh + k\alpha) \asymp 1 + \frac{F^k_1(\nh-k\alpha)}{F^k_0(\nh-k\alpha)}.
\]
The estimate \eqref{lamk.asymp} then follows from Lemma~\ref{A1A0.lemma}
(after noting that $k\brak{\alpha} = \sqrt{k^2+a^2}$).   This completes the proof in the case $k \ge N_{\vep, V}$.
\end{proof}

\bigbreak
\subsection{Low-frequency asymptotics}\label{Lfreq.sec}

The Legendre function estimates given in \S\ref{leg.asym.sec} are applicable only for $k$ sufficiently large.  
Although this covers the main region of interest, where $k$ and $s-\nh$ are comparable in magnitude, we still need to estimate for $\Lambda_k(s)$ when $k$ is small.  We could be satisfied with fairly rough estimates, since the low-frequency terms make a contribution of order $a$ to asymptotic of leading order $a^{n+1}$.  However, we need lower bounds in particular, and there is no general estimate that will provide these. 

Fortunately, the asymptotics of the Legendre functions for large $\nu$ with $k$ fixed are well-covered in the literature.  From Olver \cite[Thm~12.9.1 and \S12.12]{Olver}, we have the following:
\begin{prop}\label{PQ.IK.prop}
For $\re\nu> -\tfrac12$ and $r \in [0,r_0]$ we have
\[
\begin{split}
P^{-k}_{\nu}(\cosh r) & = \frac{1}{\nu^k} \left(\frac{r}{\sinh r}\right)^{\frac12} I_k((\nu+\tfrac12) r) 
\left(1 + O_{k,r_0}(\abs{\nu}^{-1})\right) \\
\bQ^{k}_{\nu}(\cosh r) & = \frac{\nu^k}{\Gamma(k+\nu+1)} 
\left(\frac{r}{\sinh r}\right)^{\frac12} K_k((\nu+\tfrac12)  r) \left(1 + O_{k,r_0}(\abs{\nu}^{-1})\right) 
\end{split}
\]
\end{prop}

Using standard estimates for the modified Bessel functions, we obtain the bounds
\begin{equation}\label{IK.bounds}
\abs{I_k(z)} \le C_{k,\vep} h^+_k(z),\qquad \abs{K_k(z)} \le C_{k,\vep} h^-_k(z),
\end{equation}
for $\abs{\arg z} < \tfrac{\pi}2 - \vep$ where
\[
h^\pm_k(z) := \begin{cases}\abs{z}^{\pm k} & \abs{z} \le 1, \\
\abs{z}^{-\frac12} e^{\pm \re z} & \abs{z} > 1. \end{cases}
\]
For bounded $k$ the inductive estimates are furnished by the following:
\begin{lemma}
For $s = \nh+ae^{i\theta}$ with $\abs{\theta} \le \tfrac{\pi}2-\vep$ and $a \ge N_{\vep, V}$, 
\begin{equation}\label{J.Fminus}
\int_r^{r_0} \abs{J^k(s;r,t) V(t)} h^-_k(tae^{i\theta}) (\sinh t)^{-\frac{n-1}2}\>dt \le \frac{C_{k,\vep,V}}{a} h^-_k(rae^{i\theta})(\sinh r)^{-\frac{n-1}2},
\end{equation}
and
\begin{equation}\label{J.Fplus}
\begin{split}
&\int_r^{r_0} \abs{J^k(s;r,t) V(t)} h^+_k(tae^{i\theta}) (\sinh t)^{-\frac{n-1}2}\>dt \\
&\qquad\le  \frac{C_{k,\vep,V}}{a} \Bigl[h^+_k(rae^{i\theta}) + (a \cos\theta)^{-\sigma} e^{2r_0a\cos\theta} h^-_k(rae^{i\theta}) \Bigr] 
(\sinh r)^{-\frac{n-1}2}.
\end{split}
\end{equation}
\end{lemma}
\begin{proof}
By Proposition~\ref{PQ.IK.prop} and \eqref{IK.bounds} we can estimate
\[
\abs{J^k(s;r,t)} \le C_{k,\vep} \Bigl[ h^-_k(rae^{i\theta})h^+_k(tae^{i\theta}) +  h^+_k(rae^{i\theta})h^-_k(tae^{i\theta}) \Bigr]
(\sinh r)^{-\frac{n-1}2} (\sinh t)^{\frac{n+1}2}.
\]
The estimate \eqref{J.Fminus} the follows easily from the definition of $h^\pm_k$.

Using the same estimate for $J^k(s;r,t)$ we break up the left-hand side of \eqref{J.Fplus} into two terms.  Estimation of the term with 
$h^+_k(rae^{i\theta})$ in front works just as in the estimate for \eqref{J.Fminus}.   The term with $h^-_k(rae^{i\theta})$
out front involves the integral
\[
\int_{r}^{r_0} h^+_k(tae^{i\theta})^2 \abs{V(t)} \sinh t\>dt.
\]
Using the Laplace estimate from Proposition~\ref{Laplace.prop} we have
\[
\begin{split}
\int_{\max(r,1/a)}^{r_0} h^+_k(tae^{i\theta})^2 \abs{V(t)} \sinh t\>dt 
& =  \int_{\max(r,1/a)}^{r_0} (at)^{-1}  e^{2ta\cos\theta} \abs{V(t)} \sinh t \>dt,  \\
& \le \frac{C_V}{a} (a \cos\theta)^{-\sigma} e^{2r_0a\cos\theta}.
\end{split}
\]
If $r < 1/a$ then we must also consider
\[
\int_{r}^{1/a} h^+_k(tae^{i\theta})^2 \abs{V(t)} \sinh t\>dt = \int_{r}^{1/a} (at)^{2k} \abs{V(t)} \sinh t\>dt \le \frac{C_V}{a}.
\]
For $a$ sufficiently large, we will have $(a \cos\theta)^{-\sigma} e^{2r_0a\cos\theta} \ge 1$, so that this extra term may
be combined with the Laplace estimate term.
\end{proof}

\bigbreak
\begin{proof}[Proof of Proposition~\ref{lamk.prop} (part two)]
To complete the estimate, we will show that for $k$ fixed, 
$r \in [0,r_0]$, $\abs{\theta} < \tfrac{\pi}2 - \vep$, and $\dist(ae^{i\theta},\bbZ/2)>\vep$, we have
\begin{equation}\label{lamk.lowf}
\abs{\Lambda_k(\nh+ae^{i\theta})} \asymp a^{-1-\sigma}e^{2r_0a\cos\theta} + O(a^{-1}),
\end{equation}
with constants that depend only on $k,\vep,V$.  From the definition \eqref{H.def} of $H(\alpha;r)$, we can easily check,
for $r$ fixed and $\abs{\alpha}\ge 1$, that
\[
H(\alpha; r) = 2r\re \alpha + O(1).
\]
Thus \eqref{lamk.lowf} implies \eqref{lamk.asymp} in the case where $k$ is bounded.

For this argument we set $\nu = -\frac12 + ae^{i\theta}$, and we make the assumption 
According to Proposition~\ref{PQ.IK.prop} and \ref{IK.bounds}, we have
\[
\abs{v^k_0(\nh+ae^{i\theta};r)}  \le C_{k,\vep} a^{-k} h^-_k(rae^{i\theta}) (\sinh r)^{-\frac{n-1}2}.
\]
By induction using Lemma~\ref{J.Fminus}, we obtain the estimates
\[
\abs{v^k_j(\nh+ae^{i\theta};r)} \le C_{k,\vep,V}^j a^{-k-j} h^-_k(rae^{i\theta}) (\sinh r)^{-\frac{n-1}2}.
\]
Using \eqref{leg.refl} in addition to Proposition~\ref{PQ.IK.prop} and \ref{IK.bounds} yields
\[
\abs{v^k_0(\nh+ae^{i\theta};r)}  \le C_{k,\vep} \Bigl[\abs{\beta^+_k(ae^{i\theta})} h^+_k(rae^{i\theta}) + \abs{\beta^-_k(ae^{i\theta})} 
h^-_k(rae^{i\theta}) \Bigr] (\sinh r)^{-\frac{n-1}2},
\]
where
\[
\beta^+_k(z) := \Gamma(k+z+\tfrac12)\sin(\pi z) z^{-k},
\]
and
\[
\beta^-_k(z) := \frac{z^k}{\Gamma(k-z+\tfrac12)}.
\]
The induction argument corresponding to that of Lemma~\ref{vj.bad} then gives
\[
\begin{split}
\abs{v^k_j(\nh-ae^{i\theta};r)} & \le \left(\frac{C_{k,\vep,V}}{a}\right)^j \biggl[ \abs{\beta^+_k(ae^{i\theta})} h^+_k(rae^{i\theta}) \\
&\quad + \Bigl( j \abs{\beta^+_k(ae^{i\theta})} (a\cos \theta)^{-\sigma} e^{2ar_0\cos\theta} + \abs{\beta^-_k(ae^{i\theta})} \Bigr) h^-_k(rae^{i\theta}) \biggr] (\sinh r)^{-\frac{n-1}2}.
\end{split}
\]

These estimates show that the two series $\sum v^k_j(\nh\pm ae^{i\theta})$ converge for $a$ sufficiently large.  By arguing 
as in Lemma~\ref{Aj.rem}, we find
\[
\Lambda_k(\nh+ae^{i\theta}) \asymp 1 + \frac{F^k_1(\nh-ae^{i\theta})}{F^k_0(\nh-ae^{i\theta})},
\]
for $a$ sufficiently large, 
with constants that depend only on $k, \vep$, and $V$.  As in the proof of Lemma~\ref{A1A0.lemma}, we split
\[
\frac{F^k_1(\nh-ae^{i\theta})}{F^k_0(\nh-ae^{i\theta})} = I_1 + I_2,
\]
where
\[
I_1 :=  \Gamma(k+ae^{i\theta}+\tfrac12) \Gamma(k-ae^{i\theta}+\tfrac12) \sin (\pi ae^{i\theta})
\int_0^{r_0} [P^{-k}_{-\nu}(\cosh t)]^2 V(t)\sinh t\>dt,
\]
and
\[
I_2 := \Gamma(k+ae^{i\theta}+\tfrac12) \int_0^{r_0} P^{-k}_{\nu}(\cosh t) \bQ^{k}_{\nu}(\cosh t) V(t)\sinh t\>dt.
\]
Using Proposition~\ref{PQ.IK.prop} and the well-known asymptotic
\[
I_k(z) = \frac{e^z}{(2\pi z)^\frac12} (1 + O(1/z)),
\]
valid for $\abs{\arg z} \le \frac{\pi}2$, we can reduce $I_1$ to
\[
I_1 \asymp \Gamma(k+ae^{i\theta}+\tfrac12) \Gamma(k-ae^{i\theta}+\tfrac12) \sin (\pi ae^{i\theta}) a^{-2k-1}
\int_0^{r_0} e^{2rae^{i\theta}} V(t)\>dt.
\]
Proposition~\ref{Laplace.prop} then gives
\[
I_1 \asymp \Gamma(k+ae^{i\theta}+\tfrac12) \Gamma(k-ae^{i\theta}+\tfrac12) \sin (\pi ae^{i\theta}) a^{-2k-1-\sigma} 
e^{2r_0a\cos\theta}.
\]
Applying Stirling's formula, and assuming $\dist(ae^{i\theta},\bbZ/2) > \vep$ in addition to the other hypotheses, we find
\[
I_1 \asymp a^{-1-\sigma}e^{2r_0a\cos\theta} + O(a^{-1}),
\]
with constants depending on $k,\vep,V$.

The second integral, $I_2$, is easily seen via Proposition~\ref{PQ.IK.prop} to be $O_{k,\vep,V}(1)$, so the result follows.
\end{proof}

\bigbreak
\section{Radial scattering determinant estimate}\label{rscdet.sec}

In this section we will complete the proof of Theorem~\ref{asymp.thm} by establishing an asymptotic for $\tau(s)$ in the radial potential case.   In terms of the matrix elements $\Lambda_k(s)$ defined in \eqref{tau.asymp}, the relative scattering determinant is
\[
\tau(s) = \prod_{l}  \Lambda_k(s)^{\mu_n(l)},
\]
where $k := l + (n-1)/2$ and $\mu_n(l)$ is the multiplicity \eqref{hnl.def}.  Our main goal is the following:
\begin{theorem}\label{tau.asymp}
Suppose $\tau(s)$ is the relative scattering determinant corresponding to a radial potential $V = V(r) \in L^\infty[0,r_0]$, such that $V$ is continuous near $r_0$ and satisfies
\[
V(r) \sim \kappa (r_0-r)^{\sigma-1},
\]
for some $\sigma\ge 1$.  Then assuming $\abs{\theta} < \tfrac{\pi}2 - \vep $ and $\dist(a,\bbZ/2) >\vep$
\[
\log \abs{\tau(\nh+ae^{i\theta})} = h_{r_0}(\theta) a^{n+1} + o(a^{n+1}),
\]
where $h_{r_0}(\theta)$ was defined in \eqref{indicator.def}.
\end{theorem}

The leading contribution to this estimate comes from terms for which $\Lambda_k(\nh+ae^{i\theta})$ exhibits exponential growth for large $k \asymp a$.  This is the case covered by Proposition~\ref{lamk.prop}. 
We must also account for the cases where the exponent $H(\alpha; r)$ is near zero or negative, for which Proposition~\ref{lamk.prop} gives no information.
For the exponential decay estimate we 
turn to the formula for the relative scattering matrix used in \cite{Borthwick:2010}.  
\begin{prop}\label{lamk.upper.prop}
For $\abs{\arg \alpha} < \tfrac{\pi}2 - \vep$, $\abs{k\alpha} \ge N_\vep$, and $\delta > 0$
\begin{equation}\label{lamk.upper}
\abs{\Lambda_k(\nh+k\alpha)-1} \le C_{\vep,r_0} \delta^{-4} k^{-1} e^{k H(\alpha; r_0+\delta)},
\end{equation}
where $H(\alpha;r)$ was defined in \eqref{H.def}.
\end{prop}
\begin{proof}
Set $r_j = r_0 + \tfrac13 j\delta$ for $j=1,2,3$.
Let $\psi\in\cinf(\bbR)$ be a cutoff function with $\psi(t)=0$ for $t\le 0$ and $\psi(t)=1$ for $t\ge 1$.  Then set
$\chi_j(r) = \psi((r-r_j)/\delta)$, so that $\chi_j = 1$ for $r\le r_j$ and $\chi_j = 0$ for $r \ge r_{j+1}$.
Then from the proof of \cite[Lemma~4.1]{Borthwick:2010} we have
\[
S_V(s)S_0(s)^{-1} - 1 = (2s-n) E_0(s)^t [\Delta, \chi_2] R_V(s) [\Delta,\chi_1] E_0(n-s),
\]
where $E_0(s)$ is the unperturbed Poisson operator on $\bbH^{n+1}$.  Since $V$ is radial, $S_V(s)$ is diagonalized by spherical harmonics and $\Lambda_k(s)$ may be computed as a matrix element,
\begin{equation}\label{lamk.matrix}
\begin{split}
\Lambda_k(s) - 1 &= \brak{Y_l^m, \bigl[S_V(s)S_0(s)^{-1} - 1 \bigr] Y_l^m} \\
& = (2s-n) \Bigl\langle \chr_{[r_2,r_3]} E_0(s)Y_l^m, [\Delta, \chi_2] R_V(s) [\Delta,\chi_1] \chr_{[r_1,r_2]} E_0(n-s)Y_l^m \Bigr\rangle, 
\end{split}
\end{equation}
where $\chr_{[r_i,r_{i+1}]}$ denotes the multiplication operator of the characteristic function $\chi_{[r_i,r_{i+1}]}(r)$.
As in the proof of Proposition~\ref{tau.upper}, we can derive from Proposition~\ref{res.est} the bound
\begin{equation}\label{RV.est}
\bigl\Vert [\Delta_0, \chi_2] R_V(s) [\Delta_0,\chi_1] \bigr\Vert \le C_{\vep,r_0}\delta^{-4},
\end{equation}
under the assumptions $\abs{\arg(s-\nh)} < \tfrac{\pi}2 - \vep$ and $\abs{s-\nh} > N_\vep$, which keep $s(n-s)$ bounded away from 
the spectrum of $\Delta+V$.  

Also, using the decomposition of the Poisson operator in polar coordinate from \cite[Prop.~4.2]{Borthwick:2010}, 
we can compute
\[
E_0(s)Y_l^m = e_k(s;r) Y_l^m,
\]
where
\[
e_k(s;r) := 2^{-1-\nu} \pi^{1/2}  \frac{\Gamma(l+s)}{\Gamma(s-\frac{n}2 +1)}  (\sinh r)^{-\frac{n-1}2}  P^{-k}_{\nu}(\cosh r),
\]
with $\nu = s - \frac{n+1}2$ as before.
Thus from \eqref{lamk.matrix} and the resolvent estimate \eqref{RV.est} we have the bound
\[
\begin{split}
\abs{\Lambda_k(s)} &\le C_{\vep,r_0}\delta^{-4} |2s-n| \left[ \int_{r_1}^{r_2} \abs{e_k(n-s; r)}^2\>(\sinh r)^n\>dr \right]^{\frac12} 
\\
&\qquad \left[ \int_{r_2}^{r_3} \abs{e_k(s; r)}^2\>(\sinh r)^n\>dr\right]^{\frac12}.
\end{split}
\]
Substituting in with the definition if $e_k$ and using the reflection identity $P^{-k}_\nu = P^{-k}_{-1-\nu}$, we can rewrite this as
\[
\begin{split}
\abs{\Lambda_k(s)} & \le C_{\vep,r_0}\delta^{-4} \abs{\sin(\pi k\alpha) \Gamma(k+k\alpha+\tfrac12)  \Gamma(k-k\alpha+\tfrac12)} \\
&\quad \left[ \int_{r_1}^{r_2} \abs{P^{-k}_{\nu}(\cosh r)}^2 \sinh r\>dr \right]^{\frac12} 
\left[ \int_{r_2}^{r_3} \abs{P^{-k}_{\nu}(\cosh r)}^2 \sinh r\>dr \right]^{\frac12}.
\end{split}
\]
Applying Corollary~\ref{PQ.bounds} and using the fact that $\re\phi(\alpha;\cdot)$ is increasing 
allows us to reduce this estimate to 
\[
\abs{\Lambda_k(s)} \le C_{\vep,r_0}\delta^{-4} \abs{\frac{\sin(\pi k\alpha) \Gamma(k+k\alpha+\tfrac12)  \Gamma(k-k\alpha+\tfrac12)}{\Gamma(k+1)^2}} e^{2k\re[\phi(\alpha;r_3)+p(\alpha)]}.
\]
The result then follows from \eqref{Stir.H} and the definition of $H(\alpha;r)$.
\end{proof}

\bigbreak
The final issue that we need to resolve before proving Theorem~\ref{tau.asymp} is the behavior on the zone where 
$H(\alpha;r) \approx 0$.   This region contains non-trivial zeros of $\Lambda_k(s)$, so to produce a lower bound on 
$\log \Lambda_k(s)$ is delicate.  The only tool we have for this situation is the Minimum Modulus Theorem for entire functions.  

\begin{prop}\label{inter.prop}
For any $\beta>2$ we have
\[
\log \abs{\Lambda_k(s)} \ge - c_{\beta} (k+\brak{s})(1+\log \brak{s}),
\]
under the assumption that $\dist(s, \calR_0\cup (n - \calR_V)) \ge \brak{s}^{-\beta}$.
\end{prop}
\begin{proof}
The solutions $v^k(s;r)$ are analytic as functions of $s$, so the functions $F^k(s)$ appearing in the formula
\eqref{SV.AA} for the scattering matrix element are analytic as well.  

Using Stirling, we can easily produce a crude bound,
\[
\log\abs{F^k_0(s)} \le C(k+\brak{s})(1+\log \brak{s}),
\]
for all $z\in \bbC$ with with $C$ independent of $k$.  From the estimates on the series $F^k = \sum F^k_j$ 
produced in \S\ref{Hfreq.sec} and \S\ref{Lfreq.sec},
and the straightforward bound $H(\alpha;r_0) = O_{r_0}(\abs{\alpha})$, we claim the same result holds in 
the perturbed case:
\[
\log\abs{F^k(s)} \le C_V (k+\brak{s})(1+\log \brak{s}).
\]
Here $C$ depends on $V$ but not on $k$.
(Note our bounds on $F^k_j(s)$ and $F^k_j(n-s)$ omit the sectors $\abs{\arg(s-\nh)} \in [\tfrac{\pi}2 - \vep, \tfrac{\pi}2 - \vep]$,
 but this restriction was necessary only for the lower bounds.  A simple application of Phragm\'en-Lindel\"of extends the upper bounds to the missing sectors.)

We can now apply the Minimum Modulus Theorem (see e.g.~\cite[Thm.~3.7.4]{Boas}) to obtain a corresponding lower bound: for $\beta>2$
\[
\log \abs{F^k(s)} \ge - c_{\beta} (k+\brak{s})(1+\log \brak{s}).
\]
under the restriction that $\dist(s,n-\calR_V) \ge \brak{s}^{-\beta}$.  (Here we use also the fact that for any $k$ the zeros of $F^k(s)$ are included in the set $n-\calR_V$.)

Applying these upper and lower bounds to the matrix element formula \eqref{SV.AA} yields the result.
\end{proof}

\bigbreak
\begin{proof}[Proof of Theorem~\ref{tau.asymp}]
We have already noted the more general upper bound in Proposition~\ref{tau.upper}, so our goal here is to produce a
corresponding lower bound:
\begin{equation}\label{tau.lower}
\log \abs{\tau(\nh+ae^{i\theta})}  \ge h(\theta, r_0) a^{n+1} - o(a^{n+1}).
\end{equation}

By conjugation we can assume that $\theta \in [0,\tfrac{\pi}2-\vep)$.  We need to estimate the sum
\begin{equation}\label{tau.sum2}
\log \abs{\tau(\nh+ae^{i\theta})} = \sum_{k} \mu_n(k-\tfrac{n-1}2) \log \abs{1+ \Lambda_k(\nh+ae^{i\theta})}. 
\end{equation}
From the definition \eqref{hnl.def} we can estimate the multiplicities by
\begin{equation}\label{hnk.asymp}
\mu_n(k-\tfrac{n-1}2) = \frac{2k^{n-1}}{\Gamma(n)}(1 + O(k^{-1})).
\end{equation}

For $\theta \in [0,\tfrac{\pi}2]$, let $\varrho(\theta)$ be the implicit solution of $H(\varrho(\theta) e^{i\theta}, r_0) = 0$, as shown in Figure~\ref{hplot.fig}.  
Since 
\[
\del_xH(xe^{i\theta};r_0)|_{x=\varrho(\theta)} > c_\vep,
\]
for $\theta\in [0,\tfrac{\pi}2-\vep)$, we can see that
\[
H(xe^{i\theta}; r_0) > c_\vep a^{-\frac12}\quad\text{for }x \ge \varrho(\theta)(1+a^{-\frac12}),
\]
for $a$ sufficiently large.  Thus, under the assumption $a/k \ge \varrho(\theta)(1+a^{-\frac12})$, $k H(k^{-1}ae^{i\theta};r_0)$ is bounded below by $c_\vep a^{\frac14}$ for $k \ge ca^{\frac34}$.   On the other hand, for $k < ca^{\frac34}$, we have
$a/k > a^{-\frac14}/c$, and if $a/k$ is large we can use the approximation $kH(k^{-1}ae^{i\theta};r_0) \asymp r_0a\cos\theta$.   We conclude that, under these assumptions,
\[
kH(k^{-1}ae^{i\theta};r_0) \ge c_\vep a^{\frac14},
\]
for all $a$ sufficiently large.  Thus for $a/k \ge \varrho(\theta)(1+a^{-\frac12})$ with $a$ sufficiently large and $\dist(ae^{i\theta},\bbZ/2) >\delta$,
Proposition~\ref{lamk.prop} implies
\begin{equation}\label{logl.kH}
\log \abs{\Lambda_k(\nh+ae^{i\theta})} \ge kH(k^{-1}ae^{i\theta};r_0) - O(\log a).
\end{equation}

With this estimate in mind, we divide the sum \eqref{tau.lower} into three pieces, roughly according to the sign of $H$:
\[
\begin{split}
\Sigma_+&:\> \frac{a}{k} \ge \varrho(\theta)(1+a^{-\frac12}) \\
\Sigma_0&:\> \varrho(\theta)(1-a^{-\frac12}) < \frac{a}{k} < \varrho(\theta)(1+a^{-\frac12})\\
\Sigma_-&:\> \frac{a}{k} \le \varrho(\theta)(1-a^{-\frac12}).
\end{split}
\]
The dominant term is $\Sigma_+$, and from \eqref{hnk.asymp} and \eqref{logl.kH} we obtain
\[
\Sigma_+ \ge \sum_{k \le a/[\varrho(\theta)(1+a^{-\frac12})]} 
\frac{2k^{n-1}}{\Gamma(n)} H\left(\frac{ae^{i\theta}}{k}, r_0\right) - O(a^{n}).
\]
The summand is monotonic as a function of $k$, so we can estimate with an integral:
\[
\Sigma_+ \ge \frac{2}{\Gamma(n)} \int_{0}^{a/[\varrho(\theta)(1+a^{-\frac12})]} 
k^{n+1} H\left(\frac{ae^{i\theta}}{k}, r_0\right)\>dk - O(a^n).
\]
We can then make the substitution $x = a/k$ to reduce this to
\[
\Sigma_+ \ge  \frac{2a^{n+1}}{\Gamma(n)} \int_{\varrho(\theta)(1+a^{-\frac12})}^{\infty}
\frac{H(xe^{i\theta}, r_0)}{x^{n+2}}\>dx - O(a^n).
\]
Using the fact that $H(\alpha,r_0) = O(\abs{\alpha})$ for $\abs{\alpha}$ large, we can extend the lower limit of integration
by adjusting the error term,
\begin{equation}\label{Sp.final}
\Sigma_+ \ge  \frac{2a^{n+1}}{\Gamma(n)} \int_{\varrho(\theta)}^{\infty}
\frac{H(xe^{i\theta}, r_0)}{x^{n+2}}\>dx - O(a^{n}).
\end{equation}

Next we consider the middle term $\Sigma_0$.  Proposition~\ref{inter.prop} implies that $\log\abs{1+\Lambda_k}$ is 
$O(a\log a)$ for $k$ in the range defined by $a/[\varrho(\theta)(1\pm ae^{i\theta})]$, and there are $O(a^{\frac12})$ values 
of $k$ in this range.  Hence
\[
\Sigma_0 = O(a^{n+\frac12} \log a).
\]

Finally, for $\Sigma_-$, we will use Proposition~\ref{lamk.upper.prop} which shows that $\Lambda_k$ is exponentially small in this region, together with the simple estimate,
\[
\log\abs{1+\lambda} \ge -c_\eta\abs{\lambda}\quad\text{for }\abs{\lambda}\le\eta.
\]
This yields, for small $\delta>0$,
\[
\abs{\Lambda_k(\nh+ae^{i\theta})} \le C_{\vep, r_0} \delta^{-4} k^{-1} e^{kH(k^{-1}ae^{i\theta}; r_0+\delta)}.
\]
For $k \ge a/[\varrho(\theta)(1-a^{-\frac12})]$, we have
\[
kH(k^{-1}ae^{i\theta}; r_0+\delta) \le - c_\vep \sqrt{a} + aO_\vep(\delta),
\]
so by choosing $\delta = b_\vep a^{-\frac12}$, with $b_\vep$ sufficiently small, we can bound the exponent by $-c_\vep\sqrt{a}$.  
We obtain, for $a$ sufficiently large and $\abs{\theta} < \tfrac{\pi}2 - \vep$, 
\[
\log\abs{1+\Lambda_k(\nh+ae^{i\theta})}  \ge - c e^{-c_\vep \sqrt{a}}.
\]
This yields
\[
\Sigma_- \ge -O(e^{-c_\vep \sqrt{a}})
\]

Combining \eqref{Sp.final} with these estimates on $\Sigma_0$ and $\Sigma_-$ completes the proof of \eqref{tau.lower}.
\end{proof}

\bigbreak
\begin{proof}[Proof of Theorem~\ref{asymp.thm}]
Under the hypotheses of the theorem, Theorem~\ref{tau.asymp} gives
\[
\int_{-\frac{\pi}2+\vep}^{\frac{\pi}2-\vep} \tau(\nh + ae^{i\theta}) = a^{n+1} \int_{-\frac{\pi}2+\vep}^{\frac{\pi}2-\vep} h_{r_0}(\theta)\>d\theta + o_\vep(a^{n+1}),
\]
assuming $d(a, \bbZ/2)>\vep$.
To fill the gap near $\pm \tfrac{\pi}2$, we use the general estimate from Lemma~\ref{lind.lemma}.  This gives
\[
\int_{\frac{\pi}2-\vep \le \abs{\theta} \le \frac{\pi}2}  \tau(\nh + ae^{i\theta}) = O(\vep a^{n+1}),
\]
for $a$ in some unbounded set $J \subset \bbR_+$ that depends on the choice of $\beta$ in the lemma.
We conclude that 
\[
\int_{-\frac{\pi}2}^{\frac{\pi}2} \tau(\nh + ae^{i\theta}) = a^{n+1} \int_{-\frac{\pi}2}^{\frac{\pi}2} h_{r_0}(\theta)\>d\theta 
+ O(\vep a^{n+1}) + o_\vep(a^{n+1}),
\]
for $a \in J$.
Taking $\vep\to 0$ then gives
\[
\int_{-\frac{\pi}2}^{\frac{\pi}2} \tau(\nh + ae^{i\theta}) \sim a^{n+1} \int_{-\frac{\pi}2}^{\frac{\pi}2} h_{r_0}(\theta)\>d\theta.
\]
In combination with Proposition~\ref{count.prop}, this completes the proof.
\end{proof}

\bigbreak
\section{Distribution of resonances for generic potentials}\label{dist.sec}

In this section we'll develop the theorems on resonance distribution outlined in \S\ref{intro.sec}.  These results draw on techniques from Christiansen \cite{Christ:2005, Christ:2012}, Christiansen-Hislop \cite{CH:2005}, and Borthwick-Christiansen-Hislop-Perry \cite{BCHP}.  We will only sketch the proofs for cases that are very similar to these earlier results.

Before stating the results we recall a standard definition from several complex variables.  A
\emph{pluripolar} subset $E$ of a connected domain $\Omega \subset \bbC^p$ is the polar set of a plurisubharmonic function $\psi$ on $\Omega$, i.e.~ the set  $\psi^{-1}\{-\infty\}$.  Pluripolar sets have Lebesgue measure zero in $\bbC^p$.  Moreover, for $p=1$ the real part $E \cap \bbR$ of a pluripolar set $E$ will have Lebesgue measure zero in $\bbR$.

\subsection{Asymptotics of the counting function}
The first result shows that the order of growth in the bound $N_V(t) = O(t^{n+1})$ of Theorem~\ref{upper.thm} is saturated for generic potentials.  
\begin{theorem}\label{Kgen.thm}
Let $F = \bbR$ or $\bbC$.
Given a compact subset $K\subset \bbH^{n+1}$ with non-empty interior, the set
\[
\left\{V \in L^\infty(K; F): \> \limsup_{t\to \infty} \frac{\log N_V(t)}{\log t} = n+1  \right\}
\]
is \emph{Baire typical} in $L^\infty(K; F)$ (i.e.~is a dense $G_\delta$ subset).
\end{theorem}

\begin{proof}
This is the hyperbolic analog of the main result from Christiansen-Hislop \cite{CH:2005}.  Those techniques were already adapted to
hyperbolic manifolds in \cite{BCHP}, so we will not repeat all the details here.  Suppose $V_z \in L^\infty(K; F)$ is a holomorphic family of potentials for $z\in \Omega \subset \bbC^{p}$, an open connected set.  With a relatively simple adaption of the proof of \cite[Thm.~1.1]{Christ:2005}, similar to the version given in
\cite[\S5.2]{BCHP}, we can show that if $\limsup_{t\to \infty} \log N_{V_z}(t)/\log t = n+1$ holds for some particular $z_0 \in \Omega$, then it holds for $z \in \Omega - E$, where $E$ is a pluripolar set.  
Given any $V \in L^\infty(K; \bbC)$ we can choose $V_0$ a radial potential supported in $K$ to which Theorem~\ref{asymp.thm} applies and form the family $V_z := (1-z)V_0 + zV$.
In this way we conclude that $N_{V_z}(t)$ has maximal rate of growth except for $z$ in some pluripolar set.  Then one can argue exactly as in Christiansen-Hislop \cite{CH:2005} or \cite[\S5.3]{BCHP} to characterize the class of $V$ with maximal growth rate as Baire typical.
\end{proof}

\bigbreak
Following Christiansen \cite{Christ:2012}, we can prove a variant of Theorem~\ref{Kgen.thm} involving the sharp asymptotic constant rather than just the order of growth, at the cost of restricting the supporting set $K$ to a closed ball $\overline{B}(r_0) \subset \bbH^{n+1}$.   
\begin{theorem}\label{order.thm}
Suppose $V_z(x)\in L^\infty(\bbH^{n+1},\bbC)$ is a holomorphic family of potentials for $z\in \Omega\subset \bbC^p$, an open connected subset.  
Assume that $\supp V_z \subset \overline{B}(r_0)$ for all $z$, and that the condition,
\begin{equation}\label{limsup.tN}
\limsup_{a\to\infty} \frac{\tN_{V_z}(a)}{a^{n+1}} = A_n(r_0),
\end{equation}
holds for some $z_0 \in \Omega$, where $A_n(r_0)$ is the asymptotic constant defined in \eqref{An.def}.
Then there exists a pluripolar set $E \subset \Omega$ such that \eqref{limsup.tN} holds for $z \in \Omega-E$.
\end{theorem}
\begin{proof}
The proof is closely related to the proof of \cite[Thm.~1.2]{Christ:2012}.  The only complication is that in our case $\tau_{V_z}(s) := \det S_{V_z}(s) S_0(n-s)$ has infinitely many poles on the positive real axis for $n$ odd, at the points $n-\calR_0$.  To handle this, we introduce a function $g_0$ defined as a Hadamard product
\[
g_0(s) := \prod_{\zeta \in (n-\calR_0) \cup e^{i\pi/(n+1)} \calR_0} E\left(\frac{z}{\zeta};\> n+1\right),
\]
where $E(z;p)$ denotes an elementary factor.  The extra zeros at $e^{i\pi/(n+1)} \calR_0$ are inserted so that $g_0(s)$ will have a regularly distributed zero set in the sense of Levin \cite[\S II.1]{Levin}.  By \cite[Thm~II.2]{Levin}, there is a smooth indicator function $H_0(\theta)$ such that
\begin{equation}\label{g0.ind}
\lim_{a\to\infty} \frac{\log \abs{g_0(\nh+ae^{i\theta})}}{a^{n+1}} = H_0(\theta),
\end{equation}
uniformly for $\theta \in S^1$.  We can then cancel off the extra poles of $\tau_{V_z}(s)$ by introducing
\begin{equation}\label{gzs.def}
g(z,s) := \tau_{V_z}(s) g_0(s),
\end{equation}
whose poles for $\re s\ge \nh$ correspond to the (finitely many) points of $\calR_{V_z}$ on that side.  

By Proposition~\ref{count.prop}, we then have
\[
\tN_{V_z}(a) = \Psi_1(z,a) + \Psi_2(a) + O(a^n),
\]
where 
\[
\Psi_1(z,s) :=  \frac{n+1}{2\pi} \int_{-\frac{\pi}2}^{\frac{\pi}2} \log \abs{g(z, \nh+ae^{i\theta})}\>d\theta,
\]
and 
\[
\Psi_2(a) := A^{(0)}_n a^{n+1} - \frac{n+1}{2\pi} \int_{-\frac{\pi}2}^{\frac{\pi}2} \log \abs{g_0(\nh+ae^{i\theta})}\>d\theta.
\]
By \eqref{g0.ind}, $\lim_{a\to\infty} a^{-(n+1)} \Psi_2(a)$ exists and is given by some constant $c_0<\infty$.
Theorem~\ref{upper.thm} thus gives the bound
\[
\limsup_{a\to\infty} \frac{\Psi_1(z,a)}{a^{n+1}} \le A_n(r_0) - c,
\]
and by hypothesis, the maximum is achieved at $z=z_0$.  
For $a$ sufficiently large, the function $\Psi_1(z,a)$ is plurisubharmonic in $z$, and so by Lelong-Gruman
\cite[Prop.~1.39]{LelongGruman}, there exists a pluripolar set $E \subset \Omega$ such that
\[
\limsup_{a\to\infty} \frac{\Psi_1(z,a)}{a^{n+1}} = A_n(r_0) - c,
\]  
for $z \in \Omega-E$.  This proves the claim.
\end{proof}

Note that as in Theorem~\ref{Kgen.thm}, we could use Theorem~\ref{order.thm} in conjunction with perturbation by radial potentials to show that the condition \eqref{limsup.tN} holds on a Baire typical subset of $L^\infty(\overline{B}(r_0), F)$.

\bigbreak
\subsection{Distribution in sectors}

Many of the results of Christiansen \cite{Christ:2012} concern the resonance counting function restricted to sectors.
In the hyperbolic case it is natural to center these sectors at $s = \nh$ and define:
\[
N_V(t,\theta_1,\theta_2) := \#\Bigl\{\zeta\in\calR_V:\>0<\abs{\zeta-\nh}\le t, \>\arg(\zeta-\nh) \in [\theta_1,\theta_2]\Bigr\}.
\]
The ``averaged'' sectorial counting function is denoted with a tilde:
\[
\tN_V(a,\theta_1,\theta_2) := (n+1) \int_0^a \frac{N_V(t,\theta_1,\theta_2)}{t} \>dt.
\]
The results below refer to the indicator function $h_{r_0}(\theta)$, defined in \eqref{indicator.def} and illustrated in Figure~\ref{indicator.fig}.
\begin{figure}
\begin{center}
\begin{overpic}{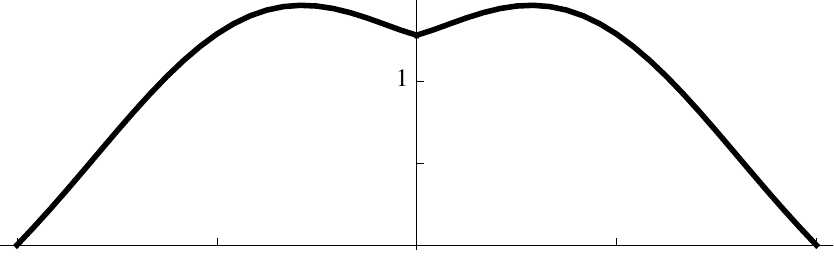}
\put(-4,0){$-\pi/2$}
\put(95,0){$\pi/2$}
\put(80,26){$h_1(\theta)$ for $n=2$}
\end{overpic}
\end{center}
\caption{The indicator function.}\label{indicator.fig}
\end{figure}

\begin{theorem}\label{Msector.thm}
Suppose $V\in L^\infty(\overline{B}(r_0),\bbC)$ has counting function satisfying $N_V(t) \sim A_n(r_0)t^{n+1}$.  
Then for $\tfrac{\pi}2 \le \theta_1 < \theta_2 \le \tfrac{3\pi}2$, with $\theta_i \ne \pi$,
\[
\begin{split}
N_V(t,\theta_1, \theta_2) &= N_0(t,\theta_1, \theta_2) 
+ \frac{(n+1) t^{n+1}}{2\pi} \int_{\theta_1-\pi}^{\theta_2-\pi} h_{r_0}(\omega)\>d\omega\\
&\qquad + \frac{t^{n+1}}{2\pi(n+1)}
\begin{cases} h_{r_0}'(\theta_2-\pi)&\text{if }  \theta_2 < \frac{3\pi}2 \\ 0 &\text{if } \theta_2 = \frac{3\pi}2 \end{cases} \\
&\qquad - \frac{t^{n+1}}{2\pi(n+1)}
\begin{cases} h_{r_0}'(\theta_1-\pi)&\text{if }  \theta_1 > \frac{\pi}2 \\ 0 &\text{if }  \theta_1 = \frac{\pi}2 \end{cases} \\
&\qquad + o(t^{n+1}).
\end{split}
\]
Note that $N_0$ contributes only if $n$ is odd, and then only if $\pi \in (\theta_1, \theta_2)$.
\end{theorem}
\begin{proof}
It suffices to prove only the case $\theta_1 = \tfrac{\pi}2$, 
\[
N_V(t,\tfrac{\pi}2, \theta) = N_0(t,\tfrac{\pi}2, \theta) + \left[\frac{h_{r_0}'(\theta-\pi)}{2\pi(n+1)}
+\frac{n+1}{2\pi} \int_{-\frac{\pi}2}^{\theta-\pi} h_{r_0}(\omega)\>d\omega\right] t^{n+1} + o(t^{n+1}).
\]
If $V$ is real, then self-adjointness implies a conjugation symmetry in $\calR_V$, and the $\theta_2 = \tfrac{3\pi}2$ case would be 
equivalent to this by reflection.  In the general case, the indicator function $h_{r_0}$
still posseses the conjugation symmetry even though  $\calR_V$ does not, 
so the proof for $\theta_2 = \tfrac{3\pi}2$ is identical to the one we will give for $\theta_1 = \tfrac{\pi}2$.  
The intermediate case $[\theta_1,\theta_2] \subset (\tfrac{\pi}2, \tfrac{3\pi}2)$ follows by subtracting the two endpoint cases.
 
To begin we apply the argument principle to the integral of $\tau'/\tau$ over a sector given by
$\abs{z-\nh} \le t$ and $-\tfrac{\pi}2 \le \arg (z-\nh) \le -\tfrac{\pi}2 + \eta$.  The result is
\[
\begin{split}
& N_V(t,\tfrac{\pi}2, \tfrac{\pi}2+\eta) - N_V(t, -\tfrac{\pi}2, -\tfrac{\pi}2+\eta) - N_0(t,\tfrac{\pi}2, -\tfrac{\pi}2+\eta) \\
&\hskip1in = \frac{1}{2\pi} \int_0^t \del_r \arg \tau(\nh - ir)\>dr 
+  \frac{1}{2\pi} \int_{-\frac{\pi}2}^{-\tfrac{\pi}2+\eta} \del_\omega \arg \tau(\nh + te^{i\omega})\>d\omega \\
&\hskip1in\qquad - \frac{1}{2\pi} \int_0^t \del_r \arg\tau(-ir e^{i\eta})\>dr.
\end{split}
\]
The term $N_V(t, -\tfrac{\pi}2, -\tfrac{\pi}2+\eta)$ is bounded by a fixed constant, the total number of discrete eigenvalues.
And for the first segment of the contour integral, on critical line $\re s = \nh$, 
we have a bound $O(t^n)$ by Proposition~\ref{sigma.bound}.
By applying the Cauchy-Riemann equations to the integrands of the two remaining integrals, we obtain
\[
\begin{split}
N_V(t,\tfrac{\pi}2, \tfrac{\pi}2+\eta) - N_0(t,\tfrac{\pi}2, \tfrac{\pi}2+\eta) 
& =  \frac{1}{2\pi} \int_{-\frac{\pi}2}^{-\tfrac{\pi}2+\eta} t \del_t \log \abs{\tau(\nh + te^{i\omega})} \>d\omega \\
&\qquad + \frac{1}{2\pi} \del_\eta J_\tau(t,-\tfrac{\pi}2+\eta) + O(t^{n}),
\end{split}
\]
where
\[
J_f(t,\omega) := \int_0^t \log \abs{f(\nh + r e^{i\omega})}\>\frac{dr}{r}.
\]
(This is a slight adaptation of the definition from Levin \cite{Levin}, moving the center to $\nh$.)
Now, to eliminate the derivatives, we divide by $t$ and then integrate over $t$ from $0$ to $a$ and over $\eta$ from $0$ to $\theta$.  Note that Proposition~\ref{sigma.bound} implies $J_\tau(t,-\tfrac{\pi}2) = O(t^n)$.  The result of these integrations is therefore that
\begin{equation}\label{tN.int}
\begin{split}
&\int_{0}^{-\frac{\pi}2+\theta} \left[\tN_V(a,\tfrac{\pi}2,  \tfrac{\pi}2+\eta) - \tN_0(a,\tfrac{\pi}2,  \tfrac{\pi}2+\eta) \right]d\eta \\
&\qquad =  \frac{n+1}{2\pi} \int_{0}^{-\frac{\pi}2+\theta} \int_{-\frac{\pi}2}^{-\tfrac{\pi}2+\eta} \log \abs{\tau(\nh + ae^{i\omega})} \>d\omega \>d\eta 
+ \frac{n+1}{2\pi} \int_0^a J_\tau(t, \theta-\pi)\>\frac{dt}{t}\\
&\qquad\qquad + O(t^{n}).
\end{split}
\end{equation}

To continue, we use the background function $g_0(s)$ introduced in the proof of 
Theorem~\ref{order.thm} to define $g(s) := \tau(s) g_0(s)$.  If $\Delta + V$ has discrete eigenvalues, then $g(s)$ will still 
have poles at a finite set $\{\zeta_1, \dots, \zeta_m\} \subset (\nh,n)$.  In this case we can simply replace 
\[
g(s) \leadsto g(s) \prod_{i=1}^m \frac{s-\zeta_i}{s-n+\zeta_i},
\]
which will remove the poles without affecting the asymptotics.  For notational convenience, we will simply assume that $g(s)$ is analytic for the rest of the proof.  

By Proposition~\ref{tau.upper} and \eqref{g0.ind} we have a bound for $\abs{\omega} \le \tfrac{\pi}2$,
\[
\frac{\log \abs{g(\nh+ae^{i\omega})}}{a^{n+1}} \le h_{r_0}(\omega) + H_0(\omega) + o(1),
\]
as $a \to \infty$, where the Maximum Modulus principle is used to remove the restriction on the values of $a$.  Thus
\[
\limsup_{a\to\infty} \frac{\log \abs{g(\nh+ae^{i\omega})}}{a^{n+1}} \le h_{r_0}(\omega) + H_0(\omega),
\]
for $\abs{\omega} \le \tfrac{\pi}2$.  The left-hand side is by definition the indicator function of $g$, so \cite[Thm I.28]{Levin} gives, for any $\vep>0$,
\begin{equation}\label{gind.bd}
\frac{\log \abs{g(\nh+ae^{i\omega})}}{a^{n+1}} \le h_{r_0}(\omega) + H_0(\omega) + \vep,
\end{equation}
for $\omega \le \tfrac{\pi}2$ and $r \ge r_\vep$.

On the other hand, by the assumption on the asymptotics of $N_V(t)$, together with Proposition~\ref{count.prop}, we have
\[
\lim_{a\to\infty} a^{-(n+1)} \int_{-\frac{\pi}2}^{\frac{\pi}2} \log \abs{\tau(\nh+ae^{i\omega})}\>d\omega = 
\int_{-\frac{\pi}2}^{\frac{\pi}2} h_{r_0}(\omega)\>d\omega.
\]
In conjunction with \eqref{gind.bd} this implies
\[
\lim_{a\to\infty}  \int_{-\frac{\pi}2}^{\frac{\pi}2} \abs{h_{r_0}(\omega) + H_0(\omega)- 
\frac{\log \abs{g(\nh+ae^{i\omega})}}{a^{n+1}}} d\omega  =0.
\]
Then, by the final argument from \cite[Thm~IV.3]{Levin} (used also in \cite[Prop.~2.2]{Christ:2012}), $g(\nh + \cdot)$
is of \emph{completely regular growth} in the angle $(-\tfrac{\pi}2, \tfrac{\pi}2)$, with indicator function equal to $h_{r_0} + H_0$.  This means that 
\[
\lim_{a\to\infty} \frac{\log \abs{g(\nh+ae^{i\omega})}}{a^{n+1}} = h_{r_0}(\omega) + H_0(\omega),
\]
uniformly for $\omega \in (-\tfrac{\pi}2, \tfrac{\pi}2)$, for $a$ outside of some subset of zero relative measure in $[0,\infty)$.
By \cite[Lemma~III.2]{Levin}, we have 
\[
\lim_{a\to \infty} (n+1) a^{-(n+1)} J_g(t,\omega) = h_{r_0}(\omega) + H_0(\omega),
\]
for $\abs{\omega} \le \tfrac{\pi}2$.  The corresponding limit holds also for $J_{g_0}$, since $g_0$ has completely regular growth by construction.  Thus
\begin{equation}\label{J.tau.lim}
\lim_{a\to \infty} (n+1) a^{-(n+1)} J_\tau(t,\omega) = h_{r_0}(\omega),
\end{equation}
for $\abs{\omega} \le \tfrac{\pi}2$.  

Returning to \eqref{tN.int}, we can apply \eqref{J.tau.lim} to obtain
\[
\begin{split}
&\lim_{a\to \infty} a^{-(n+1)} \int_{0}^{-\frac{\pi}2+\theta} 
\left[\tN_V(a,\tfrac{\pi}2,  \tfrac{\pi}2+\eta) - \tN_0(a,\tfrac{\pi}2,  \tfrac{\pi}2+\eta) \right]d\eta \\
&\qquad =  \frac{n+1}{2\pi} \int_{0}^{-\frac{\pi}2+\theta} \int_{-\frac{\pi}2}^{-\tfrac{\pi}2+\eta} h_{r_0}(\omega) \>d\omega \>d\eta 
+ \frac{1}{2\pi (n+1)} h_{r_0}(\theta-\pi).
\end{split}
\]
The integral over $\eta$ can be removed using \cite[Lemma~5.4]{Christ:2012}, provided we avoid the point $\theta=\pi$ where 
$h_{r_0}$ fails to be differentiable.  
We then complete the proof using the fact that the asymptotic $\tN_*(a) \sim ca^{n+1}$ is equivalent to $N_*(t) \sim ct^{n+1}$ (see, e.g,. \cite[Lemma~1]{Stefanov:2006}).
\end{proof}

\medbreak
One interesting feature of Theorem~\ref{Msector.thm} is that the indicator function remains the same whether $\Delta+V$ is self-adjoint or not.  Even though the conjugation symmetry is broken in the non-self-adjoint case, the resonance distribution still exhibits this symmetry in an asymptotic sense.
This situation is illustrated in Figure~\ref{plot1I.fig}, which shows the resonance plots for real and imaginary radial step potentials in $\bbH^{2}$.  The breaking of the conjugation symmetry is clear only in the vicinity of the origin.

\begin{figure}
\begin{center}
\begin{overpic}{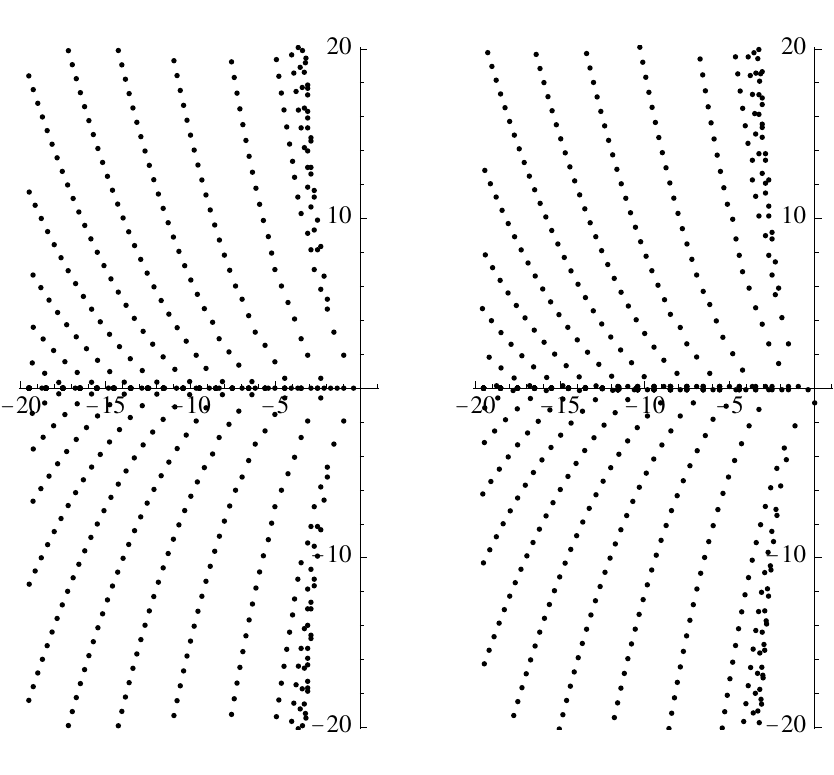}
\put(15,0){$V = \chi_{B(1)}$}
\put(70,0){$V = i\chi_{B(1)}$}
\end{overpic}
\end{center}
\caption{Comparison of resonance distributions for real and imaginary step potentials in $\bbH^{2}$.}\label{plot1I.fig}
\end{figure}

\bigbreak
\subsection{Sectorial asymptotics for generic potentials}

Finally, and once again following \cite{Christ:2012}, we present some results on distribution of resonances in sectors that hold in a generic sense.  The first result concerns the resonances in a narrow sector bordering on the critical line $\re s = \nh$, as shown in Fig~\ref{sector.fig}.  The theorem gives a lower bound independent of $\vep$ on the number of resonances in this strip that is independent of $\vep$.  As in the remark following Theorem~\ref{order.thm}, we could  use perturbation by radial potentials to show that this condition holds for a Baire typical subset of  $L^\infty(\overline{B}(r_0), F)$.  
\begin{figure}
\begin{center}
\begin{overpic}[scale=.7]{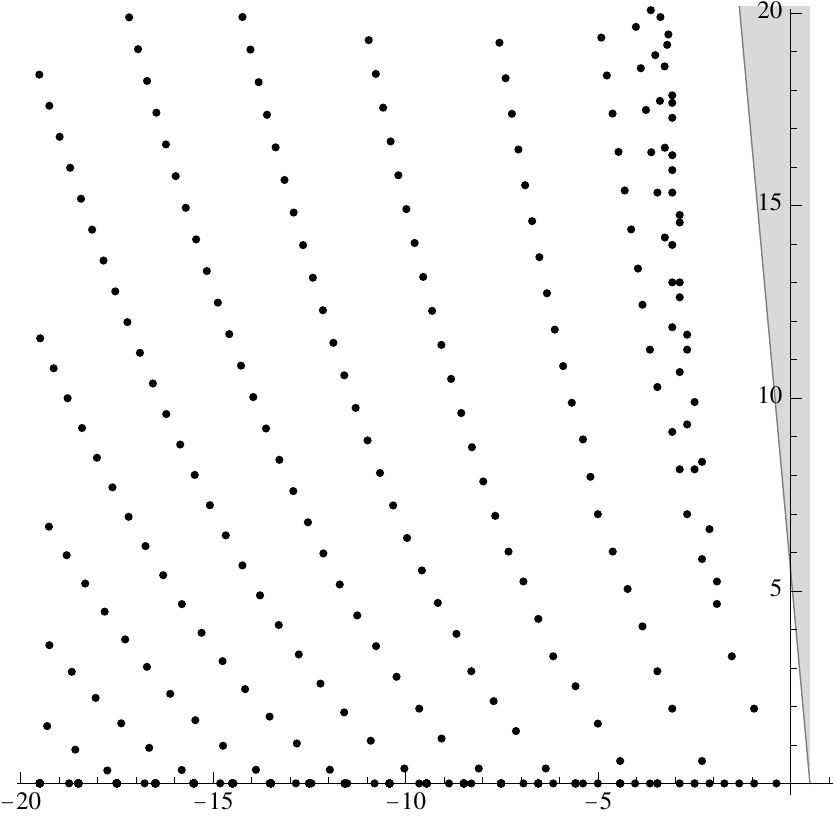}
\end{overpic}
\end{center}
\caption{An example of the type of sector to which Theorem~\ref{Ngeneric} applies.  For generic potentials, the resonance count in such a sector satisfies $N_V(t;\tfrac{\pi}2-\vep,\tfrac{\pi}2) \asymp t^{n+1}$, with constants independent of $\vep$.}\label{sector.fig}
\end{figure}

\begin{theorem}\label{Ngeneric}
Suppose $V_z(x)\in L^\infty(\bbH^{n+1},\bbC)$ is a holomorphic family of potentials for $z\in \Omega\subset \bbC^p$, an open connected subset.  
Assume that $\supp V_z \subset \overline{B}(r_0)$ for all $z$, and that the condition,
\begin{equation}\label{limsup.tN2}
\limsup_{a\to\infty} \frac{\tN_{V_z}(a)}{a^{n+1}} = A_n(r_0),
\end{equation}
holds for some $z_0 \in \Omega$, where $A_n(r_0)$ is the asymptotic constant defined in \eqref{An.def}.
Then for any $0<\vep<\tfrac{\pi}2$, there exists a pluripolar set $E_\vep$ such that
\[
\limsup_{t\to\infty} \frac{\tN_{V_z}(a,\tfrac{\pi}2, \tfrac{\pi}2+\vep)}{a^{n+1}} 
\ge  \frac{1}{4\pi (n+1)} h_{r_0}'(-\tfrac{\pi}2+)  \quad\text{for }z\in \Omega-E_\vep.
\]
The corresponding result holds in the conjugate sector also, i.e.~for $\tN_{V_z}(a,\tfrac{3\pi}2-\vep, \tfrac{3\pi}2)$.
\end{theorem}
\begin{proof}
Let us denote by $\tau_{V_z}$ the relative scattering determinant associated to $V_z$.  As in the proof of Theorem~\ref{Kgen.thm}
we introduce the background function $g_0(s)$ to cancel the poles of $\tau_{V_z}(s)$ coming from $\calR_0$ (necessary only if $n$ is odd).
Then $g(z,s) := \tau_{V_z}(s) g_0(s)$ is analytic for $\re s \ge \nh$, except for finitely many poles related to the discrete spectrum.  
As in the proof of Theorem~\ref{Msector.thm}, we note that these poles are easily cancelled off (see \cite[Lemma~5.2]{Christ:2012} for details),
and for the sake of exposition we assume that $g(z,s)$ is analytic for the rest of the proof.

From \eqref{tN.int}, assuming $\theta\in (\tfrac{\pi}2,\pi)$ we have
\[
\int_{0}^{-\frac{\pi}2+\theta} \tN_V(a,\tfrac{\pi}2,  \tfrac{\pi}2+\eta) = \Psi(z,a,\theta) - \Psi_0(a,\theta) + O(t^n),
\]
where
\[
\Psi(z,a,\theta) := \frac{n+1}{2\pi} \int_{0}^{-\frac{\pi}2+\theta} \int_{-\frac{\pi}2}^{-\tfrac{\pi}2+\eta} 
\log \abs{g(z,\nh + ae^{i\omega})} \>d\omega \>d\eta 
+ \frac{n+1}{2\pi} \int_0^a J_{g(z,\cdot)}(t, \theta-\pi)\>\frac{dt}{t},
\]
and
\[
\Psi_0(a,\theta) := \frac{n+1}{2\pi} \int_{0}^{-\frac{\pi}2+\theta} \int_{-\frac{\pi}2}^{-\tfrac{\pi}2+\eta} \log \abs{g_0(\nh + ae^{i\omega})} \>d\omega \>d\eta 
+ \frac{n+1}{2\pi} \int_0^a J_{g_0}(t, \theta-\pi)\>\frac{dt}{t}.
\]
Now we use the fact that $\Psi(z,a,\theta)$ is plurisubharmonic as a function of $z$, and argue as in the proof of 
Theorem~\ref{order.thm}.  Using the assumption on $z_0$ and \cite[Prop.~1.39]{LelongGruman}, we find that there exists a pluripolar set $E_\theta \subset \Omega$ such that
\[
\begin{split}
& \limsup_{a\to\infty} a^{-(n+1)} \left[\Psi(z,a,\theta) - \Psi_0(a,\theta)\right] \\
&\qquad = \frac{1}{2\pi(n+1)} h_{r_0}(\theta-\pi) + \frac{n+1}{2\pi}  
\int_{0}^{-\frac{\pi}2+\theta} \int_{-\frac{\pi}2}^{-\frac{\pi}2+\eta} h_{r_0}(\omega)\>d\omega\>d\eta,
\end{split}
\]
for $z\in\Omega-E_\theta$.  From this point we can simply follow the end of the proof of \cite[Thm.~1.2]{Christ:2012}, to take the limit
$\theta \to \tfrac{\pi}2+$.
\end{proof} 

We can be slightly more explicit about the constant appearing in Theorem~\ref{Ngeneric}, although it doesn't reduce to a simple formula.
We start from \eqref{indicator.def}, written as
\[
h_{r_0}(\theta) :=  \frac{2}{\Gamma(n)}  \int_{\varrho(\theta)}^\infty  
\frac{H(x e^{i\theta}; r_0)}{x^{n+2}}\>dx,
\]
where $\varrho(\theta)$ is the implicit solution of $H(\varrho(\theta) e^{i\theta}, r_0) = 0$.  We can easily compute
\[
\del_\theta H(x e^{i\theta}; r_0)\big|_{\theta = -\frac{\pi}2} = -x \log(x^2+1) + 2x \log \abs{x \cosh r_0 + \sqrt{x^2\sinh^2 r_0 - 1}}.
\]
Noting also that $\varrho(\pm \tfrac{\pi}2) = 1/\sinh r_0$, we obtain
\[
h_{r_0}'(-\tfrac{\pi}2+) = 2 \int_{1/\sinh r_0}^\infty x^{-(n+1)} \log \left( \frac{x \cosh r_0 + \sqrt{x^2\sinh^2 r_0 - 1}}{\sqrt{x^2+1}} \right) \> dx.
\]

\bigbreak
Our final result in this section concerns the ``expected value'' of the resonance counting function, computed as a weighted average over a complex family of potentials.  The following result says that such weighted averages will exhibit asymptotic behavior with optimal growth, both globally and in sectors.

\begin{theorem}\label{Naverage}
Under the hypotheses of Theorem~\ref{Ngeneric}, assume that $\psi \in C_0(\Omega)$ satisfies 
\[
\int_\Omega \psi \>dm = 1,
\]
where $m$ denotes Lebesgue measure on $\bbC^p$.  Then
\[
\int_{\Omega} N_{V_z}(t) \psi(z)\>dm(z) \sim  A_n(r_0) t^{n+1}.
\]
Furthermore, for $\tfrac{\pi}2 \le \theta_1 < \theta_2 \le \tfrac{3\pi}2$, with $\theta_i \ne \pi$,
\[
\begin{split}
\int_{\Omega} N_{V_z}(t,\theta_1, \theta_2)\>dm(z) &= N_0(t,\theta_1, \theta_2) 
+ \frac{(n+1) t^{n+1}}{2\pi} \int_{\theta_1-\pi}^{\theta_2-\pi} h_{r_0}(\omega)\>d\omega\\
&\qquad + \frac{t^{n+1}}{2\pi(n+1)}
\begin{cases} h_{r_0}'(\theta_2-\pi)&\text{if }  \theta_2 < \frac{3\pi}2 \\ 0 &\text{if } \theta_2 = \frac{3\pi}2 \end{cases} \\
&\qquad - \frac{t^{n+1}}{2\pi(n+1)}
\begin{cases} h_{r_0}'(\theta_1-\pi)&\text{if }  \theta_1 > \frac{\pi}2 \\ 0 &\text{if }  \theta_1 = \frac{\pi}2 \end{cases} \\
&\qquad + o(t^{n+1}).
\end{split}
\]
\end{theorem}

This result is the analog of Christiansen \cite[Thm.~1.3]{Christ:2012}.  The only major adjustment required in the proof is the replacement of
$\tau_{V_z}(s)$ by $g(z,s) := \tau_{V_z}(s) g_0(s)$ in the case where $n$ is odd.  Since this change was already discussed in the proofs of Theorems \ref{order.thm}, \ref{Msector.thm}, and \ref{Ngeneric}, and otherwise the details are thoroughly covered in 
\cite[\S5]{Christ:2012}, we will omit the proof.

\appendix
\section{Laplace's method}

Suppose $\phi: [a,b] \to \bbC$ is a smooth function with $\re \phi' >0$.
If $u(t)$ is smooth and non-vanishing at $b$, then the classical Laplace's method gives the asymptotic
$$
\int_a^b e^{2k \phi(t)} u(t)\>dt \sim \frac{u(b)e^{2k\phi(b)}}{2k \phi'(b)},
$$
as $k \to \infty$.   

In this appendix we will extend this classical result to include rougher assumptions on $u$ and an explicit estimate of the error.
This is fairly straightforward, but we include the details because the uniformity of the error estimate is crucial in our application.
\begin{prop}\label{Laplace.prop}
Assume that $\phi: [a,b] \to \bbC$ is a smooth function with $\re\phi'>0$, and that $u \in L^\infty[a,b]$ is continuous near $b$ and satisfies
$$
u(t) \sim A (b-t)^{\sigma-1}\quad\text{ as }t\to b,
$$
for some $\sigma \ge 1$.  For
$$
I(k) := \int_a^b e^{2k\phi(t)} u(t)\>dt,   \qquad
f(k) :=  A \frac{\Gamma(\sigma)}{(2k\phi'(b))^\sigma}\> e^{2k\phi(b)},
$$
we have $I(k) \sim f(k)$ as $k\to \infty$.

More precisely, assume that for $\beta, \vep >0$ $\phi$ satisfies the bounds
$$
\abs{\phi'(b)} \ge \beta, \quad \frac{\re\phi'(b)}{\abs{\phi'(b)}} \ge \beta, \quad \sup_{[a,b]}\abs{\phi''} \le \frac{1}{\beta}.
$$
Then, given $\delta >0$, there exists $N = N(\beta,\delta, u)$ such that
$k \ge N_{\vep,\delta}$ implies
$$
\abs{\frac{I(k)}{f(k)} - 1} \le \delta.
$$ 
\end{prop}
\begin{proof}
Let us write
$$
\int_a^b e^{2k\phi(t)} u(t)\>dt  = I_1 + I_2 + I_3,
$$
where
$$
I_1 = A \int_{b-k^{-3/4}}^b e^{2k\phi(t)} (b-t)^{\sigma-1} \>dt,
$$
$$
I_2 =  \int_{b-k^{-3/4}}^b e^{2k\phi(t)} \left( u(t) - A(b-t)^{\sigma-1} \right) \>dt.
$$
$$
I_3 = \int_a^{b-k^{-3/4}} e^{2k\phi(t)} u(t) \>dt.
$$

For the first integral, we substitute $x = b-t$ and define 
\[
h(x) = \phi(b-x) - \phi(b) + \phi'(b) x,
\]
so that
\begin{equation}\label{I1.def}
I_1 = A e^{2k\phi(b)} \int_0^{k^{-\eta}} e^{-2k\phi'(b)x} e^{kh(x)} x^{\sigma-1}\>dx.
\end{equation}
This can be expressed as a Gamma integral plus some error terms:
$$
I_1  = A e^{2k\phi(b)} \left[ \frac{\Gamma(\sigma)}{(2k\phi'(b))^\sigma} + J_1 + J_2 \right],
$$
where 
$$
J_1 = \int_0^{k^{-3/4}} e^{-2k\phi'(b)x} \left(e^{kh(x)} - 1\right) x^{\sigma-1}\>dx,
$$
and
$$
J_2 = \int_{k^{-3/4}}^\infty e^{-2k\phi'(b)x} x^{\sigma-1}\>dx.
$$

To estimate $J_1$ we use Taylor's theorem to obtain
\[
\abs{h(x)} \le  \beta^{-1} \frac{x^2}2,
\]
for $x \in [a,b]$.  In particular, for $k$ sufficiently large we have
$$
\sup_{x\in [0,k^{-3/4}]} \abs{e^{kh(x)} - 1} \le C\beta^{-1} k^{-\frac12}.
$$
Applying this estimate in $J_1$, and then replacing the upper limit in the integral by $\infty$, we obtain the estimate
$$
\abs{J_1} \le \frac{\Gamma(\sigma)}{(2k \re\phi'(b))^\sigma} C\beta^{-1}k^{-\frac12}.
$$
For the $J_2$ term we can simply use the standard estimate on an incomplete Gamma function,
$$
\abs{J_2} \le \frac{1}{2k \re \phi'(b)} e^{- 2k^{1/4} \re \phi'(b)}.
$$

Now consider the second integral, $I_2$.  Given any $\delta>0$, we will have 
$$
\abs{u(t) - A(b-t)^{\sigma-1}} \le \frac12 \delta (b-t)^{\sigma-1},
$$
for all $t$ sufficiently close to $b$.  Hence, for $k \ge N_{\delta,u}$, 
$$
\abs{I_2} \le \frac12 \delta \int_{b-k^{-3/4}}^b e^{2k\re \phi(t)} (b-t)^{\sigma-1} \>dt.
$$
By the same analysis we used on $I_1$ we find
$$
\abs{I_2} \le \frac12 \delta e^{2k\re\phi(b)} \left[ \frac{\Gamma(\sigma)}{(2k \re\phi'(b))^\sigma} 
(1 +  C\beta^{-1}k^{-\frac12})
+ \frac{1}{2k \re \phi'(b)} e^{-2k^{1/4} \re \phi'(b)} \right],
$$
for sufficiently large $k$.

The third integral is estimated for $k\ge N_\vep$ by
\[
\begin{split}
\abs{I_3} & \le b \norm{u}_\infty \exp\left[2k \re \phi(b - k^{-3/4})\right]\\
& \le  b \norm{u}_\infty e^{2k \re \phi(b)} e^{-2k^{1/4}\re \phi'(b)} \exp\left[ k^{-\frac12}/\beta \right],
\end{split}
\]
again using Taylor and the assumed bound on $\phi''$.

Collecting these estimates, we find that for $k \ge N(\vep, \delta, u)$,
\[
\abs{\frac{I(k)}{f(k)} - 1} \le \frac{\delta}2 + C_\beta \left( \frac{\abs{\phi'(b)}}{\re\phi'(b)} \right)^\sigma k^{-\frac12}
+ C_\beta k^\sigma \abs{\phi'(b)}^\sigma e^{-2k^{1/4}\re \phi'(b)},
\]
Using the lower bounds on $\re\phi'(b)/\abs{\phi'(b)}$ and $\abs{\phi'(b)}$, this becomes
\[
\abs{\frac{I(k)}{f(k)} - 1} \le \frac{\delta}2 + C_{\beta,\vep} \left( k^{-\frac12} + k^\sigma e^{-c k^{1/4}} \right).
\]
We can adjust $N(\beta, \delta, u)$ as necessary to make the right-hand side smaller than 
$\delta$ for $k \ge N(\beta,  \delta, u)$.
\end{proof}

\end{document}